\documentclass{amsart}
\usepackage{color}
\usepackage{amscd}
\usepackage{nicefrac}
\usepackage{graphicx}
\usepackage{amssymb}
\usepackage{amsfonts}
\usepackage{euscript}
\usepackage{tikz}
\usepackage[all]{xy}

\numberwithin{equation}{section}
\renewcommand{\subsection}[1]{\hspace{-\parindent}\refstepcounter{subsection}{\bf (\arabic{section}.\alph{subsection}) #1.}\addcontentsline{toc}{subsection}{\bf #1.}}

\newcommand{\appsubsection}[1]{\hspace{-\parindent}\refstepcounter{subsection}{\bf (\Alph{section}.\alph{subsection}) #1.}\addcontentsline{toc}{subsection}{\bf #1.}}
\newenvironment{nouppercase}{%
  \renewcommand{\uppercasenonmath}[1]{}}{}

\theoremstyle{plain}
\newtheorem{thm}{Theorem}[section]

\newtheorem{theorem}[thm]{Theorem}

\newtheorem{assumption}[thm]{Assumption}

\newtheorem{corollary}[thm]{Corollary}

\newtheorem{prerequisites}[thm]{Prerequisites}

\newtheorem{remark}[thm]{Remark}

\newtheorem{example}[thm]{Example}

\newtheorem{lemma}[thm]{Lemma}

\newtheorem{conjecture}[thm]{Conjecture}
\newtheorem*{claim*}{Claim} 
\newtheorem*{lemma*}{Lemma}
\newtheorem*{theorem*}{Theorem}
\newtheorem*{conjecture*}{Conjecture}
\newtheorem{application}[thm]{Application}
\newtheorem{question}[thm]{Question}

\newcommand{\bC}{{\mathbb C}}

\newcommand{\bK}{{\mathbb K}}

\newcommand{\bP}{{\mathbb P}}

\newcommand{\bZ}{{\mathbb Z}}

\newcommand{\scrA}{\EuScript A}
\newcommand{\scrB}{\EuScript B}

\newcommand{\scrE}{\EuScript E}
\newcommand{\scrF}{\EuScript F}

\newcommand{\scrK}{\EuScript K}

\newcommand{\scrO}{\EuScript O}
\newcommand{\scrP}{\EuScript P}

\newcommand{\scrS}{\EuScript S}
\newcommand{\scrT}{\EuScript T}
\newcommand{\scrU}{\EuScript U}

\newcommand{\frakg}{\mathfrak{g}}

\newcommand{\half}{{\textstyle\frac{1}{2}}}

\newcommand{\iso}{\cong}

\renewcommand{\hom}{\mathit{hom}}

\newcommand{\cornerbar}[1]{
  \tikz[baseline=(n.base)]{\node(n)[inner sep=1pt]{$#1$};
    \draw[line cap=round](n.north west)--(n.north east)--(n.south east);
  }
}
\newcommand{\roundbar}[1]{
  \tikz[baseline=(n.base)]{\node(n)[inner sep=1pt]{$#1\,$};
    \draw[line cap=round, rounded corners](n.north west)--(n.north east)--(n.south east);
  }
}
\newcommand{\Tw}{\mathit{Tw}}
\newcommand{\cornersubbar}[1]{ 
 \begin{tikzpicture}[baseline=(n.base)]
    \node(n)[inner sep=1pt]{{$\scriptstyle #1$}};
    \draw[line cap=round](n.north west)--(n.north east)--(n.south east);
  \end{tikzpicture}
}

\title[LEFSCHETZ FIBRATIONS]{\Large\larger\rm Fukaya $A_\infty$-structures associated to\\ Lefschetz fibrations. II\,\nicefrac{1}{2}}
\author{Paul Seidel}

\begin{document}
\begin{nouppercase}
\maketitle
\end{nouppercase}
\begin{abstract}
We consider a version of the relative Fukaya category for anticanonical Lefschetz pencils. There are direct connections between the behaviour of this category and  enumerative geometry: some of these are results announced here, others remain conjectural. Among the conjectural goals is a formula for the dependence of the Fukaya category of Calabi-Yau hypersurfaces on the Novikov (area) parameter.
\end{abstract}

\section{Introduction}
This paper continues a line of thought from \cite{seidel12,seidel14b}, whose aim is to relate various flavours of Fukaya categories which appear in the context of Lefschetz fibrations. Large parts of that picture remain conjectural. The emphasis here is on making it more precise, and exploring its implications. We should warn the reader that the main result will only be stated without proof. A significant amount of space is devoted to example computations, which provide evidence for the general picture.

\subsection{Fibrewise compactification\label{subsec:the-problem}}
The most general setting we work in is as follows. Take an exact symplectic Lefschetz fibration, whose base is the complex plane, and whose (smooth) fibre is a Liouville-type symplectic manifold. To this fibration, we can associate its Fukaya $A_\infty$-category (the original idea is due to Kontsevich; for a formal definition, see \cite{seidel04}). Assume now that the fibre admits a compactification, obtained by adding a smooth divisor at infinity. In fact, assume that this can be done simultaneously for all fibres, giving rise to a fibrewise compactification (properification) of the Lefschetz fibration. To this fibrewise compactification, one can associate a one-parameter formal deformation of the original Fukaya category (a version of the relative Fukaya category \cite{seidel02,seidel11b}). In fact, one can also allow a choice of ``bulk term'', which leads to a wider class of deformations with the same formal behaviour. This choice can be thought of as a formally deformed (complexified) symplectic class; our terminology follows \cite{fukaya-oh-ohta-ono11}, where that idea proved to be important in geometric applications.

\begin{question} \label{qu:1}
Can we choose the bulk term so that the formal deformation induced by fibrewise compactification is trivial?
\end{question}

A priori, nontriviality of the deformation is what one expects, since compactification adds new holomorphic discs, which change the $A_\infty$-structure. Indeed, if we were talking about relative Fukaya categories in the ordinary sense (for symplectic manifolds, rather than Lefschetz fibrations), the answer to the question above is negative in all known cases. In the Lefschetz fibration context, it is easy to cook up simple examples for which the answer is again negative (Remark \ref{th:add-puncture}). On the other hand, there are interesting examples from mirror symmetry \cite{auroux-katzarkov-orlov04} in which the deformation is trivial (without introducing a bulk term; triviality of this deformation has sometimes been thought of as a general principle within mirror symmetry, but that turns out to be incorrect).

To see what direction to take from there, it is useful to restrict temporarily to first order infinitesimal deformation theory (whereas usually, we are interested in deformations to all orders in the formal parameter). The first order deformations of any $A_\infty$-category are governed by its Hochschild cohomology. In the case of Fukaya categories of Lefschetz fibrations, a conjectural geometric interpretation of Hochschild cohomology was given in \cite{seidel00b} (there is no proof of this in the literature yet, even though proofs of closely related results have been announced \cite{perutz10, mau11}, and the question has significant overlap with \cite{abouzaid-ganatra14}). The main ingredient in this interpretation is the monodromy at infinity. It therefore makes sense to consider a class of Lefschetz fibrations which have particularly simple monodromy at infinity, namely those coming from Lefschetz pencils. To make grading issues as simple as possible, we restrict attention even more, to anticanonical Lefschetz pencils. 

\begin{theorem} \label{th:1}
Consider a Lefschetz fibration arising from an anticanonical Lefschetz pencil. Then there is a preferred class of bulk terms, determined by Gromov-Witten invariants of the compactified total space, such that the associated formal deformation of the Fukaya category of the Lefschetz fibration is trivial.
\end{theorem}

This will be made precise later (Theorem \ref{th:main} and Corollary \ref{th:main-cor}), giving a positive answer to Question \ref{qu:1} for that class of Lefschetz fibrations. By ``determined by Gromov-Witten invariants'' we mean the following: fibrations obtained from a Lefschetz pencil have natural extensions over the two-sphere. The enumerative data that appear in Theorem \ref{th:1} are invariants counting holomorphic spheres which have degree one over the base. More precisely, those Gromov-Witten invariants determine the coefficients of a nonlinear ODE \eqref{eq:fundamental-equation}: by solving that ODE (as a formal power series), one obtains the desired bulk term.

\subsection{Natural transformations}
As explained in \cite{seidel06,seidel08,seidel12}, Fukaya categories of Lefschetz fibrations always come with an additional piece of algebraic structure: a canonical natural transformation from the Serre functor to the identity functor (one expects that this is the leading term of a more complicated structure, which goes by various names, such as: ``noncommutative divisor'' \cite{seidel14b}; and in a slightly more sophisticated version, ``$A_\infty$-algebra with boundary'' \cite{seidel12} or ``pre-Calabi-Yau algebra'' \cite{kontsevich-vlassopoulos13}). The conjectures from \cite{seidel06,seidel08} (partially proved in \cite{bourgeois-ekholm-eliashberg09}) relate this natural transformation to wrapped Floer cohomology in the total space of the fibration. On the other hand, the main result from \cite{seidel12} relates the same natural transformation to Floer cohomology in the fibre.

We will now conjecturally extend the previous discussion, by including this additional datum. In the presence of a fibrewise compactification, one should not only get a deformation of the Fukaya category, but also a corresponding deformation of its canonical natural transformation. Even in cases where the deformation of the category is trivial, the natural transformation can (and in fact is expected to) deform nontrivially.

\begin{conjecture} \label{th:2}
Suppose that we are in the situation of Theorem \ref{th:1} (anticanonical Lefschetz pencil, with appropriately chosen bulk term). Then, the deformation of the canonical natural transformation is determined by Gromov-Witten invariants of the compactified total space.
\end{conjecture}

The precise formulation (Conjecture \ref{th:connection}) involves another piece of structure, discovered in \cite{seidel14b}: for an anticanonical Lefschetz pencil, the Fukaya category has a second distinguished natural transformation from the Serre functor to the identity. After fibrewise compactification, we get deformations of both natural transformations. As in Theorem \ref{th:1}, the conjecture then takes on the form of an ODE which the deformations are supposed to satisfy, and whose coefficients are Gromov-Witten invariants. Actually, this is a linear (2x2 matrix) ODE, \eqref{eq:connection-equation}, which one can think of as a formula for a connection (acting on the space of natural transformations from the Serre functor to the identity).

\subsection{The Fukaya category of the fibre}
The main result of \cite{seidel12} established a relation between the Fukaya category of a Lefschetz fibration and that of its fibre. By assuming that this carries over to the deformed version, we arrive at the following conjecture (see Section \ref{subsec:fibre} for a more thorough explanation):

\begin{conjecture} \label{th:3}
In the situation of Conjecture \ref{th:2}, consider the full subcategory of the relative Fukaya category of the fibre consisting of a basis of vanishing cycles. Then, that subcategory is defined over a subalgebra of the algebra of formal power series, which has a single generator. Moreover, that generator can be explicitly described in terms of Gromov-Witten invariants of the compactified total space.
\end{conjecture}

One should compare this to the predictions made by homological mirror symmetry \cite{kontsevich94}. Suppose that (as is the case for the classical mirror constructions) the mirror family is obtained from an algebraic family of Calabi-Yau manifolds, by restricting to a formal disc in the parameter space (centered at the relevant large complex structure limit point). Roughly speaking, the power series that appear would then be Taylor series of functions on the parameter space. In these terms, our construction singles out a particular choice of formal disc (corresponding to the choice of bulk term) which is the germ of a rational curve lying in the parameter space. However, note that Conjecture \ref{th:3} is designed to be independent of mirror symmetry considerations. It is also independent of the usual formalism involved with mirror maps (variations of Hodge structure, Yukawa couplings). Instead, it relies heavily on the fact that our Calabi-Yau manifolds are fibres of a Lefschetz pencil. In those respects, our approach differs from the results of \cite{ganatra-perutz-sheridan14}, which (under suitable assumption) characterize mirror maps entirely within the formalism of Fukaya categories of Calabi-Yau manifolds.

\subsection{Structure of the paper} 
Section \ref{sec:notation} reviews the basic ingredients: Lefschetz fibrations, Fukaya categories, Gromov-Witten invariants, and some homological algebra (additional details for the algebraic part are provided in Appendix \ref{sec:algebra}). The discussion of Gromov-Witten theory continues in Section \ref{sec:gw}, whose aim is to showcase a certain assumption on the Gromov-Witten invariants of sections, and its implications. That assumption returns in the context of Fukaya categories in Section \ref{sec:main}, which contains the announcement of the main result, together with an outline of its proof, as well as the formulations of what would conjecturally be the next steps. Sections \ref{sec:elementary}--\ref{sec:example-revisited} discuss example computations. Even though these have been separated out, they should be viewed as an integral part of the paper, since (in spite of their limitations) they provide important support for the general ideas under consideration. We therefore suggest the following combined reading order:
Sections \ref{subsec:lefschetz}--\ref{subsec:fukaya-2}, \ref{sec:elementary}, \ref{subsec:gromov-witten},\ref{subsec:formalism}, \ref{subsec:trivial}, \ref{sec:apply-theorem}, \ref{subsec:eigenvalues}, \ref{subsec:nc-divisors}, \ref{subsec:conjecture}, \ref{subsec:example-connection}, \ref{subsec:fibre}, \ref{subsec:mirror-map}.

{\em Acknowledgments.} Maxim Kontsevich pointed out that my original formulation was in contradiction with expectations from mirror symmetry; and Mohammed Abouzaid made the crucial suggestion of resolving that by a suitable choice of bulk term. I learned much of the material in Sections \ref{subsec:cubic-pencil} and \ref{subsec:cp2-pencil} from Denis Auroux. Lemma \ref{th:artin} grew out of a discussion with Michael Artin, with the assistance of Bjorn Poonen. Remark \ref{th:oberdieck} is a variant of an observation made by Georg Oberdieck. 

This work was partially supported by: the Radcliffe Institute for Advanced Study at Harvard University (through a Fellowship, and by providing a stimulating working environment); the Simons Foundation (through a Simons Investigator award); and the National Science Foundation (through NSF grants DMS-1005288 and DMS-1500954).

\section{Setup and notation\label{sec:notation}}

\subsection{Lefschetz fibrations and Lefschetz pencils\label{subsec:lefschetz}}
Consider a symplectic Lefschetz fibration with closed fibres,
\begin{equation} \label{eq:fibrewise-compactification}
\bar{\pi}: \bar{E} \longrightarrow \bC
\end{equation}
(see e.g.\ \cite{auroux03} for an exposition). Our definition of such a fibration includes the condition that, at any regular point, the restriction of the symplectic form $\omega_{\bar{E}}$ to the symplectic orthogonal complement of the fibrewise tangent space should be positive (with respect to the orientation induced from the base $\bC$). We will always assume that $c_1(\bar{E}) = 0$, and in fact, that $\bar{E}$ comes with a preferred homotopy class of trivializations of its canonical bundle (a symplectic Calabi-Yau structure). The smooth fibre $\bar{M}$ (any two smooth fibres are isomorphic) is a closed symplectic manifold, which also inherits a symplectic Calabi-Yau structure. We write $2n = \mathrm{dim}(\bar{E}) \geq 4$, and assume that $\bar{E}$ (and hence $\bar{M}$) is connected. 

One particular class of fibrations \eqref{eq:fibrewise-compactification} arises as follows. Suppose that one has a closed symplectic manifold $\cornerbar{E}$ together with a symplectic Lefschetz fibration
\begin{equation} \label{eq:extend-to-p1}
\cornerbar{\pi}: \cornerbar{E} \longrightarrow \bC P^1 = \bC \cup \{\infty\}.
\end{equation}
We assume that $c_1(\cornerbar{E})$ is Poincar{\'e} dual to a fibre, and in fact, that $\cornerbar{E}$ comes with a preferred homotopy class of isomorphisms between its canonical bundle and the pullback of $\scrO_{\bC P^1}(-1)$.
We also assume that the fibre at $\infty$ is smooth (and it is then convenient to take $\bar{M}$ to be that fibre). Setting $\bar{E} = \cornerbar{E} \setminus \bar{M}$, and restricting $\cornerbar{\pi}$ accordingly, yields a Lefschetz fibration over $\bC$, as in \eqref{eq:fibrewise-compactification} (including its symplectic Calabi-Yau structure). 

More specifically, we will work with Lefschetz fibrations relative to a fibrewise ample divisor. Let's return to the situation \eqref{eq:fibrewise-compactification}. By a fibrewise divisor we mean a properly embedded symplectic hypersurface
\begin{equation}
\delta E \subset \bar{E},
\end{equation}
such that: no critical points of $\bar\pi$ lie on $\delta E$; and the symplectic parallel transport vector fields are tangent to $\delta E$, and give rise to a flat connection for $\delta E \rightarrow \bC$. If $\delta M \subset \bar{M}$ is the part of $\delta E$ lying in one fibre, then $\delta E \iso \bC \times \delta M$ by parallel transport. The ampleness assumption says that $\delta E$ represents the symplectic form, 
\begin{equation} \label{eq:ample-e}
[\omega_{\bar{E}}] = [\delta E] \in H^2(\bar{E}). 
\end{equation}
(Our convention is that all cohomology groups have complex coefficients, unless indicated otherwise; also, we will use Poincar{\'e} duality freely, without adopting special notation for it.) By restricting $\bar\pi$ to $E = \bar{E} \setminus \delta E$, we get an exact symplectic Lefschetz fibration
\begin{equation} \label{eq:lefschetz}
\pi: E \longrightarrow \bC,
\end{equation}
whose fibre $M = \bar{M} \setminus \delta M$ is a Liouville-type symplectic manifold. We then say that \eqref{eq:fibrewise-compactification} is a fibrewise compactification of \eqref{eq:lefschetz}.

There is a counterpart for \eqref{eq:extend-to-p1}, where we assume the existence of a fibrewise divisor
\begin{equation} \label{eq:diffeo-p1}
\delta E| \subset \cornerbar{E}.
\end{equation}
There are three conditions on this. As before, parallel transport should be tangent to $\delta E|$, and give rise to a flat connection for $\delta E| \rightarrow \bC P^1$, hence to a preferred identification $\delta E| \iso \bC P^1 \times \delta M$. Next, with respect to this identification, the normal bundles to $\delta E| \subset \cornerbar{E}$ and $\delta M \subset \bar{M}$ must be related by
\begin{equation} \label{eq:fibrewise-minus-one}
\nu_{\delta E|} \iso \scrO_{\bC P^1}(-1) \boxtimes \nu_{\delta M}.
\end{equation}
Finally,
\begin{equation} \label{eq:inflation}
[\omega_{\cornersubbar{E}}] = [\delta E|] +\lambda [\bar{M}] \in H^2(\cornerbar{E}) \quad
\text{for some $\lambda$.}
\end{equation}
In this setup, we will say that \eqref{eq:lefschetz} and its fibrewise compactification \eqref{eq:fibrewise-compactification} ``arise from an anticanonical Lefschetz pencil''. The terminology is explained by the following construction. Thanks to \eqref{eq:fibrewise-minus-one}, one can blow down $\delta E| \subset \cornerbar{E}$ along the $\bC P^1$ fibres. The outcome is  a symplectic $2n$-manifold $\roundbar{E}$, carrying a Lefschetz pencil \cite{donaldson98c, donaldson02} of hypersurfaces isomorphic to $\bar{M}$, with base locus $\delta M$ (those hypersurfaces represent the first Chern class, which is also a positive multiple of the symplectic class; $\roundbar{E}$ is monotone). Conversely, given a symplectic Lefschetz pencil of anticanonical divisors on a monotone symplectic manifold, one can blow up the base locus of the pencil to obtain \eqref{eq:extend-to-p1}. Figure \ref{fig:corner} summarizes some of the notation.
\begin{figure}
\begin{center}
\begin{picture}(0,0)%
\includegraphics{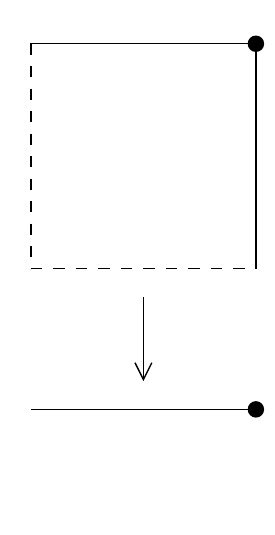}%
\end{picture}%
\setlength{\unitlength}{3552sp}%
\begingroup\makeatletter\ifx\SetFigFont\undefined%
\gdef\SetFigFont#1#2#3#4#5{%
  \reset@font\fontsize{#1}{#2pt}%
  \fontfamily{#3}\fontseries{#4}\fontshape{#5}%
  \selectfont}%
\fi\endgroup%
\begin{picture}(1455,2841)(1036,-1180)
\put(1251,1514){\makebox(0,0)[lb]{\smash{{\SetFigFont{10}{12}{\rmdefault}{\mddefault}{\updefault}{\color[rgb]{0,0,0}$\delta E \iso \bC \times \delta M$}%
}}}}
\put(1576,-736){\makebox(0,0)[lb]{\smash{{\SetFigFont{10}{12}{\rmdefault}{\mddefault}{\updefault}{\color[rgb]{0,0,0}$\bC$}%
}}}}
\put(2401,-736){\makebox(0,0)[lb]{\smash{{\SetFigFont{10}{12}{\rmdefault}{\mddefault}{\updefault}{\color[rgb]{0,0,0}$\infty$}%
}}}}
\put(1651,839){\makebox(0,0)[lb]{\smash{{\SetFigFont{10}{12}{\rmdefault}{\mddefault}{\updefault}{\color[rgb]{0,0,0}$E$}%
}}}}
\put(2476,839){\makebox(0,0)[lb]{\smash{{\SetFigFont{10}{12}{\rmdefault}{\mddefault}{\updefault}{\color[rgb]{0,0,0}$M$}%
}}}}
\put(2476,1364){\makebox(0,0)[lb]{\smash{{\SetFigFont{10}{12}{\rmdefault}{\mddefault}{\updefault}{\color[rgb]{0,0,0}$\delta M$}%
}}}}
\put(1301,-1111){\makebox(0,0)[lb]{\smash{{\SetFigFont{10}{12}{\rmdefault}{\mddefault}{\updefault}{\color[rgb]{0,0,0}Open strata}%
}}}}
\end{picture}%
\hspace{8em}
\begin{picture}(0,0)%
\includegraphics{lefschetz.pdf}%
\end{picture}%
\setlength{\unitlength}{3552sp}%
\begingroup\makeatletter\ifx\SetFigFont\undefined%
\gdef\SetFigFont#1#2#3#4#5{%
  \reset@font\fontsize{#1}{#2pt}%
  \fontfamily{#3}\fontseries{#4}\fontshape{#5}%
  \selectfont}%
\fi\endgroup%
\begin{picture}(1755,2841)(736,-1180)
\put(651,1514){\makebox(0,0)[lb]{\smash{{\SetFigFont{10}{12}{\rmdefault}{\mddefault}{\updefault}{\color[rgb]{0,0,0}$\delta E| \iso \bC P^1 \times \delta M$}%
}}}}
\put(2101,-736){\makebox(0,0)[lb]{\smash{{\SetFigFont{10}{12}{\rmdefault}{\mddefault}{\updefault}{\color[rgb]{0,0,0}$\infty$}%
}}}}
\put(2176,1364){\makebox(0,0)[lb]{\smash{{\SetFigFont{10}{12}{\rmdefault}{\mddefault}{\updefault}{\color[rgb]{0,0,0}$\delta M$}%
}}}}
\put(1001,-1111){\makebox(0,0)[lb]{\smash{{\SetFigFont{10}{12}{\rmdefault}{\mddefault}{\updefault}{\color[rgb]{0,0,0}Closed strata}%
}}}}
\put(2176,839){\makebox(0,0)[lb]{\smash{{\SetFigFont{10}{12}{\rmdefault}{\mddefault}{\updefault}{\color[rgb]{0,0,0}$\bar{M}$}%
}}}}
\put(1351,839){\makebox(0,0)[lb]{\smash{{\SetFigFont{10}{12}{\rmdefault}{\mddefault}{\updefault}{\color[rgb]{0,0,0}$\cornerbar{E}$}%
}}}}
\put(1276,-736){\makebox(0,0)[lb]{\smash{{\SetFigFont{10}{12}{\rmdefault}{\mddefault}{\updefault}{\color[rgb]{0,0,0}$\bC P^1$}%
}}}}
\end{picture}%
\end{center}
\caption{\label{fig:corner}}
\end{figure}

\subsection{Fukaya categories of Lefschetz fibrations\label{subsec:fukaya}}
To any exact symplectic Lefschetz fibration \eqref{eq:lefschetz} one can associate its Fukaya $A_\infty$-category $\scrF(\pi)$. The Calabi-Yau condition in our framework ensures that this is $\bZ$-graded. We will use complex coefficients for Floer cohomology, so that $\scrF(\pi)$ is defined over $\bC$. A fibrewise compactification gives rise to a formal deformation $\scrF_q(\bar\pi)$ of $\scrF(\pi)$, which is still $\bZ$-graded, and defined over $\bC[[q]]$. This is a version of the relative Fukaya category; the formal variable $q$ counts the intersection number of pseudo-holomorphic maps with $\delta E$. 

There are twisted versions $\scrF_b(\pi)$ of $\scrF(\pi)$, associated to a choice of bulk term
\begin{equation} \label{eq:b0-field}
b \in H^2(E;\bC^*).
\end{equation}
Very roughly speaking, for $A \in H_2(E)$, $b \cdot A \in \bC^*$ is a weight with which pseudo-holomorphic curves in class $A$ are counted. Similarly, there are twisted versions $\scrF_{q,\bar{b}}(\bar\pi)$ of $\scrF_q(\bar\pi)$, this time depending on 
\begin{equation} \label{eq:b1-field}
\bar{b} \in H^2(\bar{E};\bC[[q]]^\times).
\end{equation}
The choices \eqref{eq:b0-field} and \eqref{eq:b1-field} are related as follows: $\scrF_{q,\bar{b}}(\bar\pi)$ is a formal deformation of $\scrF_b(\pi)$, where $b$ is obtained from $\bar{b}$ by setting $q = 0$ and then restricting to $E$. Again, the bulk term determines weights, which are now given by 
\begin{equation} \label{eq:q1}
q^{\delta E \cdot A} (\bar{b} \cdot A).
\end{equation}
Reparametrizations of $q$ have a natural interpretation within this framework. Namely, suppose that we take $\scrF_{q,\bar{b}}(\bar\pi)$ and apply a parameter change 
\begin{equation} \label{eq:reparametrize-q}
q \longmapsto \beta(q), \;\; \beta(0) = 0, \; \beta'(0) \neq 0.
\end{equation}
The outcome is isomorphic to $\scrF_{q,\bar{b}_\beta}(\bar\pi)$, where
\begin{equation} \label{eq:change-of-bulk}
\bar{b}_\beta(q) = \bar{b}(\beta(q)) \cdot (\beta(q)/q)^{[\delta E]}.
\end{equation}
The first term in \eqref{eq:change-of-bulk} is obtained by applying the parameter change to the multiplicative group $\bC[[q]]^\times$ in which $\bar{b}$ takes its values; the second term maps $A \in H_2(\bar{E})$ to $(\beta(q)/q)^{\delta E \cdot A} \in \bC[[q]]^\times$. In particular, if $H^2(\bar{E})$ is one-dimensional (and hence generated by $[\delta E]$), then any $\scrF_{q,\bar{b}}(\bar\pi)$ is just a reparametrized version of $\scrF_q(\bar\pi)$. It is helpful to think of $q$ (including its appearance in the bulk term) as the variable in a ``symplectic class''
\begin{equation} \label{eq:log-area}
[\omega_{\bar{E},q,\bar{b}}] = -\log(q) [\delta E] - \log(\bar{b}).
\end{equation}
At least formally, \eqref{eq:q1} can then be rewritten in the more familiar form
\begin{equation} \label{eq:q-area}
q^{\delta E \cdot A} (\bar{b} \cdot A) = \exp\big(\!-\!\textstyle\int_A \omega_{\bar{E},q,\bar{b}}\big).
\end{equation}
We will often encounter the derivative of \eqref{eq:log-area}, namely
\begin{equation} \label{eq:derivative-area}
-\!\partial_q [\omega_{\bar{E},q,\bar{b}}] = q^{-1}[\delta E] + (\partial_q \bar{b})/\bar{b} \in H^2(\bar{E}) \otimes q^{-1}\bC[[q]].
\end{equation}
At this point, we want to restrict the generality a little, since that will simplify some of the expressions later on. For expository reasons, we have introduced \eqref{eq:b0-field} before \eqref{eq:b1-field}. However, we will assume throughout the rest of the paper that \eqref{eq:b0-field} is trivial, and correspondingly specialize \eqref{eq:b1-field} to 
\begin{equation} \label{eq:bulk}
\bar{b} \in H^2(\bar{E};1+q\bC[[q]]) \subset H^2(\bar{E};\bC[[q]]^\times).
\end{equation}
As a consequence, only parameter changes \eqref{eq:reparametrize-q} with $\beta'(0) = 1$ will be allowed.


Instead of working with the entire Fukaya categories, we usually find it more practical to fix a basis of Lefschetz thimbles $\{L_1,\dots,L_m\}$, and to consider the resulting full $A_\infty$-subcategories
\begin{align}
& \scrA \subset \scrF(\pi), \label{eq:a-subcategory} \\
& \scrA_{q,\bar{b}} \subset \scrF_{q,\bar{b}}(\bar\pi). \label{eq:aq-subcategory}
\end{align}
It is known \cite{seidel04} that the embedding \eqref{eq:a-subcategory} induces an equivalence of derived categories. As a consequence, the restriction map on Hochschild cohomology,
\begin{equation} \label{eq:restrict-hh}
\mathit{HH}^*(\scrF(\pi),\scrF(\pi)) \longrightarrow \mathit{HH}^*(\scrA,\scrA),
\end{equation}
is an isomorphism. This implies that the deformation $\scrF_{q,\bar{b}}(\bar\pi)$ is determined up to isomorphism by $\scrA_{q,\bar{b}}$.

\subsection{Fukaya categories of fibres\label{subsec:fukaya-2}}
Continuing in the previous setup, our basis of Lefschetz thimbles gives rise to a collection of vanishing cycles $\{V_1,\dots,V_m\}$ in the fibre $M$. Let 
\begin{equation} \label{eq:full}
\scrB \subset \scrF(M)
\end{equation}
be the full $A_\infty$-subcategory of the Fukaya category of the fibre formed by these cycles. Up to quasi-isomorphism, $\scrA$ can be identified with the directed subcategory of $\scrB$. If we make that identification on the nose, this means that
\begin{equation} \label{eq:directed}
\mathit{hom}_{\scrA}(L_i,L_j) = \begin{cases} 
\mathit{hom}_{\scrB}(V_i,V_j) & i<j, \\
\bC \cdot e_{L_i} \text{ (strict identity morphism)} & i=j, \\
0 & i>j,
\end{cases}
\end{equation}
with the $A_\infty$-structure of $\scrA$ being determined by that of $\scrB$ (and the requirement that the $e_{L_i}$ should be strict units). This can be useful for computing $\scrA$, since passing to the fibre reduces dimensions by $1$. One can also use the relationship in reverse, and consider $\scrA$ to be a first step towards determining $\scrB$. 

A similar relationship holds after fibrewise deformation. Consider the analogue of \eqref{eq:full} inside the relative Fukaya category of $(\bar{M},\delta M)$, twisted by the restriction of $\bar{b}$:
\begin{equation} \label{eq:bq-subcategory}
\scrB_{q,\bar{b}|\bar{M}} \subset \scrF_{q,\bar{b}|\bar{M}}(\bar{M}).
\end{equation}

\begin{lemma} \label{th:directed}
$\scrB_{q,\bar{b}|\bar{M}}$ has zero curvature (the $\mu^0$ term vanishes), and $\scrA_{q,\bar{b}}$ is quasi-isomorphic to its directed $A_\infty$-subcategory.
\end{lemma}

Strictly speaking, $\scrB_{q,\bar{b}|\bar{M}}$ is well-defined only up to isomorphism (of deformations of $\scrB$). The statement says that there are representatives with vanishing $\mu^0$. For $n \neq 2$, the existence of such a representative is clear from homological algebra (the objects are spheres, and one can arrange that they have no endomorphisms of degree $2$ on the cochain level; see \cite{fooo} for the general formalism). For $n = 2$, it follows from basic considerations involving holomorphic discs, as in \cite{seidel03b}. Only for such representatives can the directed $A_\infty$-subcategory be defined as in \eqref{eq:directed}.

\begin{remark} \label{th:split-generation}
In the case where the Lefschetz fibration comes from an anticanonical Lefschetz pencil, it is known \cite[Proposition 19.7]{seidel04} that $\scrB$ split-generates $\scrF(M)$. The analogous question for \eqref{eq:bq-subcategory} comes in several different versions. To take one of them, let 
\begin{equation} \label{eq:algebraic-closure}
\overline{\bC((q))} \supset \bC((q))
\end{equation}
be the algebraic closure of $\bC((q))$, obtained by adjoining roots $q^{1/d}$ of arbitrary order. There is a version of the Fukaya category of $\bar{M}$ (no longer relative to $\delta M$) defined over $\overline{\bC((q))}$ \cite{fooo}, and a $\bar{b}|\bar{M}$-twisted generalization. Then, that category is split-generated by its full $A_\infty$-subcategory $\scrB_{q,\bar{b}} \otimes_{\bC[[q]]} \overline{\bC((q))}$. The proof uses the long exact sequence from \cite{oh11}; see also \cite[Corollary 9.6]{seidel03b} for the four-dimensional case (both results are stated for trivial $\bar{b}$, but should generalize easily).
\end{remark}

\subsection{Gromov-Witten invariants\label{subsec:gromov-witten}}
For $A \in H_2(\cornerbar{E};\bZ)$ and $d \geq 0$, consider the Gromov-Witten invariant counting genus $0$ curves with $d$ marked points in class $A$, which is a map
\begin{equation} \label{eq:z-map}
\langle \cdots \rangle_A: H^*(\cornerbar{E})^{\otimes d} \longrightarrow \bC.
\end{equation}
The number $\langle x_1,\dots,x_d \rangle_A$ is nonzero only if the degrees satisfy
\begin{equation} \label{eq:z-map-degree}
|x_1| + \cdots + |x_d| = 2n + 2(d-3) + 2\bar{M} \cdot A.
\end{equation}
The contribution of constant maps ($A=0$) is
\begin{equation}
\langle x_1,\dots, x_d \rangle_0 = \begin{cases} \int_{\cornersubbar{E}} x_1x_2x_3 & d = 3, \\
0 & d \neq 3. \end{cases}
\end{equation}
The ``fundamental class axiom'' and ``divisor axiom'' (these go back to \cite{kontsevich-manin94} in the algebro-geometric context) say that for $A \neq 0$,
\begin{align} 
\label{eq:fundamental-class-axiom}
& \langle 1,x_2,\dots,x_d \rangle_A = 0, \\
\label{eq:divisor-axiom}
& \langle x_1,\dots,x_d \rangle_A =(x_1 \cdot A) \langle x_2,\dots, x_d \rangle_A \quad \text{if $|x_1| = 2$.}
\end{align}
Since \eqref{eq:z-map} counts stable maps representing $A$, it is zero unless either $A = 0$ or $\int_A \omega_{\cornersubbar{E}} > 0$ (this is sometimes called the ``effectiveness axiom''). Note that by deforming the symplectic form, one can make the constant $\lambda$ in \eqref{eq:inflation} arbitrarily large. Hence,
\begin{equation} \label{eq:effectiveness}
\langle \cdots \rangle_A = 0 \;\; \left\{
\begin{aligned}
& \text{if }\bar{M} \cdot A < 0, \\
& \text{or if } \bar{M} \cdot A = 0, \; \delta E| \cdot A \leq 0, \text{ and } A \neq 0.
\end{aligned}
\right.
\end{equation}
The final statement we need is more specific to our geometric situation:
\begin{equation} \label{eq:exceptional-sphere}
\text{If $\delta E| \cdot A< 0$, then $\langle \cdots \rangle_A = 0$ unless $A$ lies in the image of $H_2(\delta E|) \rightarrow H_2(\cornerbar{E})$.}
\end{equation}
Similar (actually more general) properties of Gromov-Witten invariants of blowups were derived in the algebro-geometric context in \cite{gathmann01, hu00, bayer05}.

When summing over classes $A$, we use a formal parameter $q$ as before, and allow a bulk term
\begin{equation} \label{eq:cornerbulk}
\cornerbar{b} \in H^2(\cornerbar{E};1+q\bC[[q]]) \subset H^2(\cornerbar{E};\bC[[q]]^\times).
\end{equation}
Concretely, this means that we set
\begin{equation} \label{eq:a-b-sum}
\langle x_1,\dots,x_d \rangle_{q,\cornersubbar{b}} = \sum_A  q^{\delta E| \cdot A} (\cornerbar{b} \cdot A) \, \langle x_1,\dots,x_d \rangle_A  \in \bC((q)).
\end{equation}
In analogy with \eqref{eq:log-area}--\eqref{eq:derivative-area}, one can formally write the weights in \eqref{eq:a-b-sum} as
\begin{equation} \label{eq:counting-curves}
q^{\delta E| \cdot A} (\cornerbar{b} \cdot A) = \exp\big(\textstyle\!-\!\int_A \omega_{\cornersubbar{E},q,\cornersubbar{b}}\big),
\end{equation}
where
\begin{equation} \label{eq:subbar}
[\omega_{\cornersubbar{E},q,\cornersubbar{b}}] = -\log(q)[\delta E|] - \log(\cornerbar{b})
\end{equation}
and hence
\begin{equation} \label{eq:derivative-corner}
-\!\partial_q [\omega_{\cornersubbar{E},q,\cornersubbar{b}}] = q^{-1}[\delta E|] + (\partial_q \cornerbar{b}) / \cornerbar{b} \in H^*(\cornerbar{E}) \otimes q^{-1}\bC[[q]].
\end{equation}
One difference with respect to \eqref{eq:q-area}, is that $[\delta E|]$ is not actually the cohomology class of the symplectic form, which is why negative powers of $q$ can occur in \eqref{eq:a-b-sum}. However, the difference $[\delta E|]-[\omega_{\cornersubbar{E}}]$ is a multiple of the first Chern class $[\bar{M}]$, by \eqref{eq:inflation}, and that is enough to ensure the $q$-adic convergence of \eqref{eq:a-b-sum}. As a consequence of the divisor axiom, we have 
\begin{equation} \label{eq:insertion}
\langle \!\partial_q [\omega_{\cornersubbar{E},q,\cornersubbar{b}}], x_1, \dots, x_d \rangle_{q,\cornersubbar{b}} + \partial_q  \langle x_1,\dots,x_d \rangle_{q,\cornersubbar{b}} = 0 \quad \text{for $d \neq 2$.}
\end{equation}
The class $[\bar{M}]$ also plays a distinguished role, because of \eqref{eq:z-map-degree}: if $I$ is the grading operator
\begin{equation} \label{eq:grading-operator}
I(x) = \textstyle{\frac{|x|}{2}} x,
\end{equation}
then
\begin{equation} \label{eq:scaling}
\begin{aligned}
& \langle [\bar{M}], x_1,\dots,x_d \rangle_{q,\cornersubbar{b}} + (n+d-3) \langle x_1,\dots, x_d\rangle_{q,\cornersubbar{b}} \\ & \qquad = 
\textstyle \langle I x_1,\dots,x_d \rangle_{q,\cornersubbar{b}} + \cdots +
\langle x_1,\dots, I x_d \rangle_{q,\cornersubbar{b}} \quad \text{for $d \neq 2$.}
\end{aligned}
\end{equation}
In both \eqref{eq:insertion} and \eqref{eq:scaling}, $d = 2$ is excluded because of the contribution of constant maps ($A = 0$) to the left hand side. Usually, we will rewrite the Gromov-Witten invariants (Poincar{\'e} dually) as multilinear maps
\begin{equation} \label{eq:z-multilinear}
z_{q,\cornersubbar{b}}: H^*(\cornerbar{E})((q))^{\otimes d-1} \longrightarrow H^*(\cornerbar{E})((q)).
\end{equation}
We further break them up into graded pieces $z^{(k)}_{q,\cornersubbar{b}}$, by restricting the summation to classes $A$ with $\bar{M} \cdot A = k$. 

In the simplest case $d = 1$, we get a class
\begin{equation} \label{eq:z-class}
z_{q,\cornersubbar{b}} \in H^*(\cornerbar{E})((q)).
\end{equation}
Because of \eqref{eq:effectiveness}, this has only three nontrivial graded components, which have the following more precise form:
\begin{align}
\label{eq:z0} & z^{(0)}_{q,\cornersubbar{b}} \in H^4(\cornerbar{E}) \otimes q\bC[[q]], \\
\label{eq:z1} & z^{(1)}_{q,\cornersubbar{b}} \in q^{-1} [\delta E|] + H^2(\cornerbar{E})[[q]], \\
\label{eq:z2} & z^{(2)}_{q,\cornersubbar{b}} \in H^0(\cornerbar{E};\bC[[q]]).
\end{align}
Geometrically, \eqref{eq:z0} counts pseudo-holomorphic curves in the fibres. Again due to \eqref{eq:effectiveness}, only positive powers of $q$ appear in it. Next, \eqref{eq:z1} counts ``sections''. Due to \eqref{eq:exceptional-sphere}, the only classes $A$ that can contribute with negative powers of $q$ come from $H_2(\delta E|) \iso H_2(\bC P^1) \oplus H_2(\delta M)$. Using \eqref{eq:fibrewise-minus-one}, one sees that the only such contribution comes from the generator of $H_2(\bC P^1)$, which geometrically means from ``trivial sections''
\begin{equation} \label{eq:trivial-section}
\bC P^1 \times \{\mathit{point\}} \subset \bC P^1 \times \delta M = \delta E| \subset \cornerbar{E}.
\end{equation}
This explains the form of the leading order term in \eqref{eq:z1}. Finally, \eqref{eq:z2} counts ``bisections'', and a similar argument shows that no negative powers of $q$ appear in it.

The special case $d = 3$ of \eqref{eq:z-multilinear} defines the (small) quantum product, usually written as
\begin{equation} \label{eq:quantum-product}
x_1 \ast x_2 = z_{q,\cornersubbar{b}}(x_1,x_2).
\end{equation}
As a consequence of the general WDVV relation in Gromov-Witten theory, this is associative. More specifically for our situation, the divisor axiom implies that
\begin{equation}
\label{eq:m-squared}
[\bar{M}] \ast [\bar{M}] = z^{(1)}_{q,\cornersubbar{b}} + 4 z^{(2)}_{q,\cornersubbar{b}}.
\end{equation}

There is a relation between the quantum product structure and counting holomorphic discs (with Lagrangian boundary conditions). Let $L \subset E$ be a closed Lagrangian submanifold, which is exact, graded, and {\em Spin}. This is an object of the Fukaya category of $E$, which is a full subcategory of $\scrF(\pi)$. Suppose that $L$ admits a deformation to an unobstructed object of the Fukaya category of $\bar{E}$ relative to $\delta E$, twisted by $\bar{b}$, which is a full subcategory of $\scrF_{q,\bar{b}}(\bar{\pi})$. Then, by counting holomorphic discs with boundary on $L$ which go (exactly once) through the fibre at infinity, we get an invariant
\begin{equation} \label{eq:disc-counting}
W_{L,q,\cornersubbar{b}} \in H^0(\mathit{hom}_{\scrF_{q,\bar{b}}(\bar\pi)}(L,L)) = \mathit{HF}^0_{\bar{E},\bar{b}}(L,L) = \bC[[q]].
\end{equation}
The constant ($q^0$) term counts Maslov index $2$ discs in $(\roundbar{E},L)$, as in \cite{auroux07}. Conversely, one can view \eqref{eq:disc-counting} as an extension of the idea from \cite{auroux07}, which includes higher Maslov index holomorphic discs in $\roundbar{E}$ with tangency conditions to the base locus of the pencil. The (Poincar{\'e} dual) class $[L] \in H^n(\cornerbar{E})$ satisfies
\begin{align} \label{eq:eigenvalue-1}
& [\bar{M}] \ast [L] = W_{L,q,\cornersubbar{b}} \, [L], \\
& (\partial_q[\omega_{\cornersubbar{E},q,\cornersubbar{b}}]) \ast [L] + (\partial_q W_{L,q,\cornersubbar{b}}) \, [L] = 0. \label{eq:eigenvalue-2}
\end{align}
These statements are consequences of the formalism of open-closed string maps (see \cite[Section 3.8]{fooo}, and in particular \cite[Equation (3.8.13.2)]{fooo}). The second one could be thought of as an analogue of \eqref{eq:insertion} for disc-counting.

\begin{remark} \label{th:bar-e-lagrangian}
Although that is technically more difficult, it may be possible to extend \eqref{eq:disc-counting} to more general closed Lagrangian submanifolds in $\bar{E}$. As in Remark \ref{th:split-generation}, let's consider the version of the Fukaya category of $\bar{E}$ whose coefficient field is \eqref{eq:algebraic-closure}. Nontrivial objects of that category should have disc-counting invariants $W_{L,q,\cornersubbar{b}} \in \overline{\bC((q))}$, and \eqref{eq:eigenvalue-1}, \eqref{eq:eigenvalue-2} ought to generalize accordingly (as a consequence, if $[L]$ is nonzero, $W_{L,q,\cornersubbar{b}}$ would have to lie in $\bC((q))$ after all, because it can be expressed in terms of Gromov-Witten invariants).
%
\end{remark}

\subsection{Homological algebra\label{subsec:nc-divisors}}
Take an $A_\infty$-algebra $\scrA$ over $\bC$. A noncommutative divisor on $\scrA$ \cite[Section 2c]{seidel14} is given by an $\scrA$-bimodule $\scrP$ which is invertible (with respect to tensor product, up to quasi-isomorphism), together with an $A_\infty$-algebra structure on
\begin{equation} \label{eq:b-from-a}
\scrB = \scrA \oplus \scrP[1], 
\end{equation}
such that: $\scrA \subset \scrB$ is an $A_\infty$-subalgebra; and if we consider $\scrB$ as an $\scrA$-bimodule, then the induced bimodule structure on $\scrP[1] = \scrB/\scrA$ is the previously given one. The first new piece of information, beyond $\scrA$ and $\scrP$, contained in the noncommutative divisor is a bimodule map (which we call the leading order term of the divisor)
\begin{equation} \label{eq:theta-class}
\theta \in H^0(\mathit{hom}_{(\scrA,\scrA)}(\scrP, \scrA)).
\end{equation}
Here, $(\scrA,\scrA)$ stands for the $A_\infty$ (in fact dg) category of $\scrA$-bimodules. One can think of \eqref{eq:theta-class} as the connecting map in the short exact sequence of $\scrA$-bimodules
\begin{equation}
0 \rightarrow \scrA \longrightarrow \scrB \longrightarrow \scrP[1] \rightarrow 0.
\end{equation}
More explicitly, the components of a cochain representative of $\theta$ are obtained as follows:
\begin{equation}
\scrA^{\otimes j} \otimes \scrP[1] \otimes \scrA^{\otimes i} \hookrightarrow
\scrB^{\otimes i+j+1} \xrightarrow{\mu^{i+j+1}_{\scrB}} \scrB[2-d] \xrightarrow{\text{projection}}
\scrA[2-d].
\end{equation}
It is useful to introduce a weight grading on $\scrB$, with respect to which $\scrA$ has weight $0$ and $\scrP$ has weight $-1$. The entire $A_\infty$-structure $\mu_{\scrB}^*$ is nondecreasing with respect to weights. The part that strictly respect weights consists precisely of the $A_\infty$-algebra structure of $\scrA$, and the bimodule structure of $\scrP$. The leading order term is part of the next (weight $1$) piece. There is also, for fixed $(\scrA,\scrP)$, a notion of isomorphism of noncommutative divisors. Two isomorphic divisors determine isomorphic $A_\infty$-algebra structures on $\scrB$.

\begin{remark} \label{th:rescale}
One can rescale a given noncommutative divisor by multiplying the weight $i$ part with $\lambda^i$, for some nonzero constant $\lambda$. This does not (generally) preserve its isomorphism class as a noncommutative divisor, but it preserves the isomorphism class of $\scrB$ as an $A_\infty$-algebra.
\end{remark}

\begin{lemma}[\protect{\cite[Lemma 2.8]{seidel14b}}] \label{th:theta-is-enough}
Suppose that the following holds:
\begin{equation} \label{eq:no-negative}
\text{For all $i \geq 1$, $H^*(\hom_{(\scrA,\scrA)}(\scrP^{\otimes_{\scrA} i}, \scrA))$ vanishes if $* < 0$.}
\end{equation}
Then the isomorphism class of a noncommutative divisor is entirely determined by \eqref{eq:theta-class}. \label{th:classification-theory-of-divisors}
\end{lemma}

There is also a corresponding theory of formal deformations of noncommutative divisors (where $\scrA$ and $\scrP$ remain fixed). This is given by a deformation $\scrB_q$ of the $A_\infty$-algebra structure on $\scrB$, which satisfies the same properties as before (and in particular has vanishing $\mu^0$ term). Such a deformation has a leading order term
\begin{equation} \label{eq:theta-q-class}
\theta_q \in H^0(\mathit{hom}_{(\scrA,\scrA)}(\scrP,\scrA))[[q]].
\end{equation}
The counterpart of Lemma \ref{th:theta-is-enough} is:

\begin{lemma} \label{th:q-determines}
If \eqref{eq:no-negative} holds, the isomorphism class of a deformation of noncommutative divisors is entirely determined by \eqref{eq:theta-q-class}. 
\end{lemma}

In fact, one can retain some control over the $q$-series that appear. Given a $\bC$-linear subspace $V \subset \bC[[q]]$, let 
\begin{equation} \label{eq:v-subalgebra}
S^i[V] \subset \bC[[q]]
\end{equation}
be its $i$-th symmetric product as a subspace, which means those formal power series in $q$ that can be written as homogeneous degree $i$ polynomials in the elements of $V$. Similarly, let $S[V] \subset \bC[[q]]$ be the direct sum of all the \eqref{eq:v-subalgebra}. We say that a deformation of noncommutative divisors is ``defined over $S[V]$'' if, for each $i$, the weight $i$ part of the $A_\infty$-structure of $\scrB_q$ lies in 
\begin{equation}
\Big(\prod_d \mathit{Hom}(\scrB^{\otimes d},\scrB)\Big) \otimes S^i[V].
\end{equation}

\begin{lemma} \label{th:v-restriction}
Suppose that \eqref{eq:no-negative} holds, and that we have a deformation of noncommutative divisors whose leading order term \eqref{eq:theta-q-class} lies in $H^0(\mathit{hom}_{(\scrA,\scrA)}(\scrP,\scrA)) \otimes V$, for some $V$. Then, within the same isomorphism class, there is a deformation which is defined over $S[V]$.
\end{lemma}

This result and the previous two are simple applications of abstract deformation theory; we refer to the Appendix for an explanation from that point of view.

\begin{corollary} \label{th:divide-out}
Suppose that \eqref{eq:no-negative} holds, and that we have a deformation of noncommutative divisors whose leading order term is of the form
\begin{equation} \label{eq:theta-rho-sigma}
\theta_q = t_1(q) \rho + t_2(q) \sigma,
\end{equation}
where $\rho,\sigma \in H^0(\hom_{(\scrA,\scrA)}(\scrP,\scrA))$, and $t_1,t_2 \in \bC[[q]]$, with $t_1(0) \neq 0$. Then, there is an $A_\infty$-algebra over a polynomial ring $\bC[t]$, which becomes quasi-isomorphic to $\scrB_q$ after the change of variables 
\begin{equation} \label{eq:t-variable}
t(q) = t_2(q)/t_1(q).
\end{equation}
\end{corollary}

\begin{proof}
There is an analogue of Remark \ref{th:rescale} for deformations of noncommutative divisors: without changing the $A_\infty$-algebra structure of $\scrB_q$, one can replace \eqref{eq:theta-rho-sigma} with $\rho + t(q) \sigma$. Take $V \subset \bC[[q]]$ to be the $\bC$-linear subspace spanned by $1$ and $t(q)$. Lemma \ref{th:v-restriction} yields a deformation of noncommutative divisors isomorphic to $\scrB_q$, and such that only polynomials in $t$ occur as coefficients in its $A_\infty$-structure.
\end{proof}

As a final remark, note that even though we have formulated all of the notions above for $A_\infty$-algebras, they carry over to $A_\infty$-categories in an obvious way.

\section{Counting sections\label{sec:gw}}

\subsection{The fundamental assumption\label{subsec:formalism}}
Within Gromov-Witten theory, the class \eqref{eq:derivative-corner} plays a special role. We want to consider a situation in which that class (plus a multiple of the first Chern class) can itself be written as a Gromov-Witten invariant. The resulting discussion is fairly straightforward, but in it we will encounter some objects and formulae that will be key to the categorical arguments later on. 

\begin{assumption} \label{th:fundamental-assumption}
There are functions $\psi(q) \in 1 + q\bC[[q]]$ and $\eta(q) \in \bC[[q]]$, such that
\begin{equation} \label{eq:fundamental-equation}
-\partial_q [\omega_{\cornersubbar{E},q,\cornersubbar{b}}] = q^{-1}[\delta E|] + (\partial_q \cornerbar{b})/\cornerbar{b} = \psi(q) z^{(1)}_{q,\cornersubbar{b}} - \eta(q) [\bar{M}].
\end{equation}
\end{assumption}

One way to familiarize oneself with this condition is to analyze the formal structure of \eqref{eq:fundamental-equation}, thought of as an equation for the triple $(\cornerbar{b},\psi,\eta)$. It admits a large group of symmetries, generated by transformations
\begin{equation}
\label{eq:symmetry-1}
\big(\cornerbar{b}_\alpha,\psi_\alpha,\eta_\alpha\big) =
\big(\cornerbar{b} \cdot {\alpha(q)}^{[\bar{M}]}, \alpha(q)^{-1}\psi(q), \eta(q) - \alpha(q)^{-1} \alpha'(q)\big), 
\quad \alpha \in 1 + q\bC[[q]],
\end{equation}
where $\alpha'$ is the derivative in $q$-direction, and
\begin{equation}
\label{eq:symmetry-2}
\big(\cornerbar{b}_\beta,\psi_\beta,\eta_\beta\big) =
\big(\cornerbar{b}(\beta(q))\, (\beta(q)/q)^{[\delta E|]},
\psi(\beta(q)) \beta'(q),\, \eta(\beta(q)) \beta'(q) \big), \quad
\beta \in q + q^2\bC[[q]].
\end{equation}
Note that by using \eqref{eq:symmetry-1}, one can always reduce to the case where $\psi = 1$. After that reduction, the remaining symmetry \eqref{eq:symmetry-2} has the form
\begin{equation} \label{eq:symmetry-3}
\big(\cornerbar{b}_\beta,\eta_\beta\big) = \big( \cornerbar{b}(\beta(q))  (\beta(q)/q)^{[\delta E|]}
\beta'(q)^{[\bar{M}]}, 
\, \eta(\beta(q))\beta'(q) - \beta'(q)^{-1} \beta''(q) \big).
\end{equation}
An order-by-order analysis of \eqref{eq:symmetry-3} shows that there is always a unique choice of such a transformation which reduces the situation to $\eta = 0$.

\begin{lemma} \label{th:solutions}
The equation \eqref{eq:fundamental-equation} always has a solution, which moreover is unique up to the symmetries \eqref{eq:symmetry-1} and \eqref{eq:symmetry-2}.
\end{lemma}

\begin{proof}
In view of the previous discussion, we consider only solutions with $\psi = 1$ and $\eta = 0$. We also find it easier to replace the multiplicative class $\cornerbar{b}$ with its additive counterpart $B = \log(\cornerbar{b}) \in H^2(\cornerbar{E};q\bC[[q]])$. This turns \eqref{eq:fundamental-equation} into
\begin{equation} \label{eq:b-equation}
\partial_q B = \sum_A q^{\delta E| \cdot A} e^{B \cdot A} z_A - q^{-1}[\delta E|].
\end{equation}
Here and below, the sum is only over classes $A$ with $\bar{M} \cdot A = 1$. Let's think of \eqref{eq:b-equation} as a sequence of equations for the coefficients $B = \sum_{j \geq 1} B_j q^j$. These equations have the form
\begin{equation} \label{eq:left-difference}
j B_j = \sum_A (B_j \cdot A) z_A + \text{\it (terms depending only on $B_1,\dots,B_{j-1}$)},
\end{equation}
where this time, the sum is only over classes $A$ with $\bar{M} \cdot A = 1$ and $\delta E| \cdot A = -1$. Those classes are precisely the ``trivial sections'' of \eqref{eq:trivial-section}, and hence
\begin{equation}
\sum_A (B_j \cdot A) \, z_A = C(B_j),
\end{equation}
where
\begin{equation} \label{eq:c-term}
\begin{aligned}
& C: H^*(\cornerbar{E}) \xrightarrow{\text{restriction}} H^*(\delta E|) = H^*(\bC P^1 \times \delta M) \\
& \qquad \xrightarrow{\text{Kunneth projection}} H^{*-2}(\delta M) \xrightarrow{\text{Kunneth inclusion}} H^{*-2}(\bC P^1 \times \delta M) \\ & \qquad = H^{*-2}(\delta E|) \xrightarrow{\text{pushforward}} H^*(\cornerbar{E}).
\end{aligned}
\end{equation}
Because of \eqref{eq:fibrewise-minus-one}, this endomorphism satisfies $C^2 = -C$, and therefore $j-C$ is invertible for any $j \geq 1$. After rewriting \eqref{eq:left-difference} as 
\begin{equation} \label{eq:ode3}
(j-C) B_j = \text{\it (terms depending only on $B_1,\dots,B_{j-1}$)},
\end{equation}
one can solve it recursively in $j$, and the solution is unique.
\end{proof}

We finish these preliminary considerations by looking at an important special case, namely \eqref{eq:fundamental-equation} with trivial bulk term.

\begin{lemma} \label{th:invariant-subspace}
Consider diffeomorphisms of $\cornerbar{E}$ which preserve the deformation class of the symplectic form, and which map $[\delta E|]$ to itself; they also automatically preserve $[\bar{M}] = c_1(\cornerbar{E})$. Let
\begin{equation} \label{eq:invariant-h}
H^*(\cornerbar{E})^{\mathit{inv}} \subset H^*(\cornerbar{E})
\end{equation}
be the subspace on which all such diffeomorphisms act trivially. If that subspace is spanned by $[\delta E|]$ and $[\bar{M}]$, there is a solution of \eqref{eq:fundamental-equation} with $\cornerbar{b} = 1$.
\end{lemma}

\begin{proof}
The solvability of \eqref{eq:fundamental-equation} for $\cornerbar{b} = 1$ is equivalent to the statement that
\begin{equation} \label{eq:two-d-subspace}
z^{(1)}_q \in q^{-1}\bC[[q]] \cdot [\delta E|] \oplus \bC[[q]] \cdot [\bar{M}] \subset H^2(\cornerbar{E}) \otimes q^{-1}\bC[[q]]
\end{equation}
(since the bulk term is trivial, we've omitted it from the notation). Because Gromov-Witten invariants are symplectic deformation invariants, $z^{(1)}_q$ lies in \eqref{eq:invariant-h}.
\end{proof}

\begin{lemma} \label{th:4d}
Suppose that $2n = \mathrm{dim}(E) = 4$, and that there is a solution of \eqref{eq:fundamental-equation} with $\cornerbar{b} = 1$. Then the functions $\eta$, $\psi$ appearing in that solution satisfy
\begin{equation}
\eta = -(\partial_q \psi)/\psi.
\end{equation}
\end{lemma}

\begin{proof}
Start with \eqref{eq:fundamental-equation} and take the intersection product with $[\bar{M}]$. This yields
\begin{equation}
q^{-1}\psi(q)^{-1} = \sum_A q^{\delta E| \cdot A} \langle \rangle_A,
\end{equation}
where $\langle \rangle_A$ is the numerical Gromov-Witten invariant in class $A$ (as before, $\bar{M} \cdot A = 1$). Similarly, taking the intersection product with $[\delta E|]$ yields
\begin{equation}
(\eta(q) - q^{-1})\psi(q)^{-1} = \sum_A (\delta E| \cdot A) q^{\delta E| \cdot A} \langle \rangle_A.
\end{equation}
Comparing those two gives $(\eta(q) - q^{-1})\psi(q)^{-1} = q\partial_q (q^{-1} \psi(q)^{-1})$, and hence the desired relation.
\end{proof}

\subsection{Quantum eigenvalues\label{subsec:eigenvalues}}
Within the general formalism, the $q$-derivative of Gromov-Witten invariants is described by \eqref{eq:insertion}. When combined with Assumption \ref{th:fundamental-assumption}, this leads to a nonlinear ODE for certain Gromov-Witten invariants. Specifically, we want to consider the operator of quantum multiplication with the first Chern class:
\begin{equation} \label{eq:c-operator}
\begin{aligned}
& Q_{q,\cornersubbar{b}}: H^*(\cornerbar{E})((q)) \longrightarrow H^*(\cornerbar{E})((q)), \\
& Q_{q,\cornersubbar{b}}(x) = [\bar{M}] \ast x.
\end{aligned}
\end{equation}

\begin{lemma} \label{th:xy-relation}
Under Assumption \ref{th:fundamental-assumption}, \eqref{eq:c-operator} satisfies the equation, with $I$ as in \eqref{eq:grading-operator},
\begin{equation}
 \partial_q Q_{q,\cornersubbar{b}} = (\psi(q) Q_{q,\cornersubbar{b}}^2 - \eta(q) Q_{q,\cornersubbar{b}}) (I + 1) -
I ( \psi(q) Q_{q,\cornersubbar{b}}^2 - \eta(q) Q_{q,\cornersubbar{b}}) 
- \psi(q) 4z^{(2)}_{q,\cornersubbar{b}}.
\end{equation}
\end{lemma}

\begin{proof}
To simplify the notation, all subscripts will be omitted. Using \eqref{eq:insertion} and \eqref{eq:scaling}, write
\begin{equation} \label{eq:diff-x}
\begin{aligned}
& (\partial_q Q) (x)
 = z([\bar{M}], -\partial_q [\omega],x) 
= \textstyle \sum_k k z^{(k)}(-\partial_q[\omega],x) \\
& \quad = (-\partial_q [\omega]) \ast (I + 1)x - I (-\partial_q[\omega] \ast x) \\
& \quad = (\psi z^{(1)} - \eta [\bar{M}]) \ast (I + 1)x - I( (\psi z^{(1)} - \eta [\bar{M}])\ast x) \\
& \quad = \psi z^{(1)} \ast (I + 1)x  - \eta Q(I+1)x - \psi I (z^{(1)} \ast x) +
\eta I Q x.
\end{aligned}
\end{equation}
Because of the associativity of the quantum product and \eqref{eq:m-squared}, we have
\begin{equation}
z^{(1)} \ast x = Q^2(x) - 4z^{(2)} x.
\end{equation}
Plugging that into \eqref{eq:diff-x} yields the desired result.
\end{proof}

Let $\Theta$ be the fundamental solution of the linear differential equation
\begin{equation} \label{eq:fundamental-solution}
(\partial_q + \Gamma)\,\Theta = 0, \quad
\Gamma =
\begin{pmatrix} 0 & \psi(q) \\ \psi(q) 4z^{(2)}_{q,\cornersubbar{b}} & \eta(q) \end{pmatrix}.
\end{equation}
By definition, $\Theta \in \mathit{GL}_2(\bC[[q]])$ is a matrix with constant term equal to the identity. At next order, 
\begin{equation} \label{th:first-order-theta}
\Theta = \mathit{Id} - q\Gamma(0) + O(q^2) = \begin{pmatrix} 1 & -q \\ 0 & 1-q\eta(0) \end{pmatrix} + O(q^2).
\end{equation}
We want to think of $\Theta$ (or rather its image in $\mathit{PGL}_2$) as a formal family of conformal transformations of the projective line. The next statement shows that these transformations track the $q$-dependence of the eigenvalues of \eqref{eq:c-operator}.
%
%

\begin{lemma} \label{th:eigenvalues}
Suppose that Assumption \ref{th:fundamental-assumption} holds. Then the eigenvalues of \eqref{eq:c-operator} are of the form, for $[r:s] \in \bC P^1$,
\begin{equation} \label{eq:lambda-transformation}
\lambda = \frac{\Theta_{21}r + \Theta_{22}s}{\Theta_{11}r + \Theta_{12}s} =
\begin{cases} 
\frac{s}{r} + O(q) & r \neq 0, \\
-\frac{1}{q} + \eta(0) + O(q) & r = 0.
\end{cases}
\end{equation}
\end{lemma}

\begin{proof}
As before, we work in abbreviated notation. For any $(r,s) \in \bC^2 \setminus \{(0,0)\}$, consider the matrix
\begin{equation} \label{eq:r-matrix}
R = (\Theta_{11}r + \Theta_{12}s)Q - (\Theta_{21}r + \Theta_{22}s).
\end{equation}
By definition, 
\begin{equation}
\det(R) = (\Theta_{11}r+\Theta_{12}s)^{\mathrm{rank} H^*(\cornersubbar{E})} \det(Q - \lambda),
\end{equation}
with $\lambda$ as in \eqref{eq:lambda-transformation}. We know that $\Theta_{11}r + \Theta_{12}s = r - sq + O(q^2)$ is always nonzero. For generic choice of $(r,s)$, $\lambda$ will not be one of the finitely many eigenvalues of $Q$. Let's temporarily assume that $(r,s)$ has been chosen in this way, so that $\det(R) \neq 0$.

From Lemma \ref{th:xy-relation}, it follows that
\begin{equation}
\partial_q R = R (\psi Q - \eta) (I + 1) - I (\psi Q - \eta) R + \lambda\psi(I\, R - R \, I) 
\end{equation}
Therefore,
\begin{equation} \label{eq:diff-det}
\begin{aligned}
& \det(R)^{-1} \, \partial_q \det(R) = \mathrm{tr}(R^{-1} \partial_q R) = 
\mathrm{tr}((\psi Q - \eta)(I+1) - I(\psi Q - \eta)) \\ & \quad
 = \mathrm{tr}(\psi Q - \eta) = \psi\, \mathrm{tr}(Q) - \eta\, \mathrm{rank} H^*(\cornerbar{E}).
\end{aligned}
\end{equation}
Let's rewrite this as an equality of homogeneous polynomials in the variables $r,s$:
\begin{equation} \label{eq:diff-det-2}
\partial_q \det(R) = \det(R) (\psi\, \mathrm{tr}(Q) - \eta\, \mathrm{rank} H^*(\cornerbar{E})) \in \bC((q))[r,s].
\end{equation}
The equation $\det(R) = 0$ defines a $0$-dimensional subscheme of the projective line over the algebraically closed field $\Lambda$ from \eqref{eq:algebraic-closure}. Property \eqref{eq:diff-det-2} shows that this subscheme is invariant under differentiation in $q$-direction, hence is necessarily obtained from a subscheme of $\bC P^1$ by extending constants to $\Lambda$. But by
definition, the subscheme is obtained by pulling back the eigenvalues of $Q$ by the projective transformation \eqref{eq:lambda-transformation}.
\end{proof}

Through \eqref{eq:eigenvalue-1}, Lemma \ref{th:eigenvalues} has implications for the disc-counting invariant \eqref{eq:disc-counting}, at least in the case where $[L] \in H^n(\cornerbar{E})$ is nontrivial. There is in fact a parallel argument which applies directly to \eqref{eq:disc-counting}, without any assumptions on the homology class of $L$. We will not attempt to fully work it out here, but we want to give an outline. The first step is an analogue of the divisor axiom, which says that $\partial_q W_{L,q,\cornersubbar{b}}$ counts holomorphic discs with one point going through a cycle representing $-\partial_q[\omega_{\cornersubbar{E},q,\cornersubbar{b}}]$. Using Assumption \ref{th:fundamental-assumption}, one turns this into the (nonlinear scalar first order) ODE
\begin{equation} \label{eq:scalar-quadratic}
\partial_ q W_{L,q,\cornersubbar{b}} = \psi(q) W_{L,q,\cornersubbar{b}}^2 - \eta(q) W_{L,q,\cornersubbar{b}} - \psi(q) 4 z^{(2)}_{q,\cornersubbar{b}}.
\end{equation}
By a computation similar to that in Lemma \ref{th:eigenvalues}, the solutions of this equation in $\bC[[q]]$ are
\begin{equation} \label{eq:w-values}
W_{L,q,\cornersubbar{b}} = \frac{\Theta_{21} + \Theta_{22}s}{\Theta_{11} + \Theta_{12}s} = s + O(q), \quad s \in \bC.
\end{equation}

\section{Deformations of the Fukaya category\label{sec:main}}

\subsection{The main result\label{subsec:trivial}}
Still in the same geometric situation, let's return to the Fukaya category $\scrA$ and its deformation $\scrA_{q,\bar{b}}$. The main theorem announced in this paper is:

\begin{theorem} \label{th:main}
Suppose that $\bar{b}$ is the restriction of some $\cornerbar{b}$ satisfying Assumption \ref{th:fundamental-assumption}. Then $\scrA_{q,\bar{b}}$ is a trivial deformation.
\end{theorem}

\begin{corollary} \label{th:main-cor}
There is a distinguished class of bulk terms $\bar{b}$, all related to each other by reparametrizations \eqref{eq:change-of-bulk}, such that $\scrA_{q,\bar{b}}$ is a trivial deformation.
\end{corollary}

The Corollary follows from Theorem \ref{th:main} and Lemma \ref{th:solutions}: changes of $\cornerbar{b}$ as in \eqref{eq:symmetry-1} does not affect $\bar{b}$, and those from \eqref{eq:symmetry-2} give precisely \eqref{eq:change-of-bulk}. Note that we are not claiming a converse (we are not claiming that this is the only choice of $\bar{b}$ for which $\scrA_{q,\bar{b}}$ is trivial).

Let's give an outline of the proof of Theorem \ref{th:main}. At first, we will work more generally with Lefschetz fibrations which have fibrewise compactifications \eqref{eq:fibrewise-compactification}. As a general feature of deformation theory, $\scrA_{q,\bar{b}}$ comes with a canonical cohomology class, the Kaledin class (see \cite{kaledin07,lunts10} or Appendix \ref{sec:algebra})
\begin{equation} \label{eq:dq-class}
[\partial_q \mu^*_{\scrA_{q,\bar{b}}}] \in \mathit{HH}^2(\scrA_{q,\bar{b}},\scrA_{q,\bar{b}}).
\end{equation}
What's relevant for us is the following statement, which gives a linear criterion for an a priori nonlinear problem (triviality of $A_\infty$-deformations):

\begin{lemma}[\protect{Kaledin, Lunts \cite[Proposition 4.5 and Proposition 7.3(b)]{lunts10}}] \label{th:deformation-is-trivial}
$\scrA_{q,\bar{b}}$ is a trivial $A_\infty$-deformation if and only if \eqref{eq:dq-class} vanishes. 
\end{lemma}

The next step is to replace the algebraic class \eqref{eq:dq-class} with a geometric one. Write
\begin{equation} \label{eq:q-homology}
H^*(\bar{E})_q = q^{-1} H^*(\bar{E}, E) \oplus H^*(\bar{E})[[q]].
\end{equation}
This is a graded $\bC[[q]]$-module, where multiplication by $q$ maps the first factor to the second one (by the standard map from relative to absolute cohomology). In particular, we have a distinguished class
\begin{equation} \label{eq:q-inverse-class}
q^{-1}[\delta E] \in q^{-1} H^2(\bar{E},E) \subset H^2(\bar{E})_q.
\end{equation}
We need \eqref{eq:q-homology} for a suitable version of the closed-open string map, which has the form
\begin{equation} \label{eq:basic-co}
\mathit{CO}_{q,\bar{b}}: H^*(\bar{E})_q \longrightarrow \mathit{HH}^*(\scrA_{q,\bar{b}},\scrA_{q,\bar{b}}).
\end{equation}
In analogy with other kinds of relative Fukaya categories \cite{sheridan11b}, one has:

\begin{lemma} \label{th:symplectic-class}
The Kaledin class can be expressed as follows:
\begin{equation}
[\partial_q \mu^*_{\scrA_{q,\bar{b}}}] = \mathit{CO}_{q,\bar{b}}(-[\partial_q \omega_{\bar{E},q,\bar{b}}]).
\end{equation}
\end{lemma}

Before continuing, it is useful to review a bit more background material. The map \eqref{eq:basic-co} is an analogue of the standard open-closed string map
\begin{equation} \label{eq:0-open-closed}
\mathit{CO}: H^*(E) \longrightarrow \mathit{HH}^*(\scrA,\scrA).
\end{equation}
As already pointed out in \cite[Section 6]{seidel00b}, one does not expect this to be an isomorphism. More recently, this issue has received considerable attention \cite{seidel14b,abouzaid-ganatra14}. For us, the crucial aspect is that there is a commutative diagram
\begin{equation} \label{eq:0-diagram}
\xymatrix{
H^*(E) \ar[rr] \ar[d]^-{\mathit{CO}} && \mathit{HF}^*(\phi) \ar[d]
\\
\mathit{HH}^*(\scrA,\scrA) \ar[rr]^-{\iso} && H^*(\hom_{(\scrA,\scrA)}(\scrA,\scrP)).
}
\end{equation}
Here, $\phi$ is a specific automorphism of $E$ (rotation at infinity), and $\mathit{HF}^*(\phi)$ its fixed point Floer cohomology. On the open string side, the automorphism induces an autoequivalence of $\scrF(\pi)$, and $\scrP$ is the $\scrA$-bimodule associated to that autoequivalence. The top horizontal map in \eqref{eq:0-diagram} is a kind of continuation map, and the same geometric mechanism gives rise to a bimodule homomorphism $\scrA \rightarrow \scrP$, which turns out to be a quasi-isomorphism. The bottom horizontal map in \eqref{eq:0-diagram} is induced by that homomorphism, hence is an isomorphism. As a consequence, any class in $H^*(E)$ which gets mapped to zero in $\mathit{HF}^*(\phi)$ must lie in the kernel of $\mathit{CO}$. Let's suppose that our Lefschetz fibration comes from anticanonical Lefschetz pencils. Then, the top horizontal map in \eqref{eq:0-diagram} fits into a long exact sequence
\begin{equation} \label{eq:monodromy-les}
\cdots \rightarrow H^{*-2}_{\mathit{cpt}}(M) \longrightarrow H^*(E) \longrightarrow \mathit{HF}^*(\phi) \rightarrow \cdots
\end{equation}
The conclusion is that any class in $H^*(E)$ which comes from $H^{*-2}_{\mathit{cpt}}(M)$ lies in the kernel of the open-closed string map. Of course, in degree $2$ this is a rather pointless observation, since $H^0_{\mathit{cpt}}(M) = 0$.

We now consider what happens under fibrewise compactification. There is an extension of $\phi$ to an automorphism $\bar\phi$ of $\bar{E}$. One can associate to it a relative version of fixed point Floer cohomology, denoted by $\mathit{HF}^*_{q,\bar{b}}(\bar\phi)$, as well as a bimodule $\scrP_{q,\bar{b}}$ over $\scrA_{q,\bar{b}}$. The appropriate version of \eqref{eq:0-diagram} looks as follows:
\begin{equation} \label{eq:master-diagram}
\xymatrix{
H^*(\bar{E})_q \ar[rr] \ar[d]^-{\mathit{CO}_{q,\bar{b}}} && \mathit{HF}^*_{q,\bar{b}}(\bar\phi) \ar[d]
\\
\mathit{HH}^*(\scrA_{q,\bar{b}},\scrA_{q,\bar{b}}) \ar[rr]^-{\iso} && H^*(\mathit{hom}_{(\scrA_{q,\bar{b}},\scrA_{q,\bar{b}})}(\scrA_{q,\bar{b}},\scrP_{q,\bar{b}})).
}
\end{equation}
If the image of $-[\partial_q \omega_{\bar{E},q,\bar{b}}]$ in $\mathit{HF}^*_{q,\bar{b}}(\bar\phi)$ is zero, then the deformation $\scrA_{q,\bar{b}}$ must be trivial, because of Lemmas \ref{th:deformation-is-trivial} and \ref{th:symplectic-class}. At this point, we re-introduce the assumption that the Lefschetz fibration comes from an anticanonical Lefschetz pencil, and also require that $\bar{b}$ is the restriction of some \eqref{eq:cornerbulk}. The counterpart of \eqref{eq:monodromy-les} says that the top horizontal arrow from \eqref{eq:master-diagram} fits into a long exact sequence of graded $\bC[[q]]$-modules
\begin{equation} \label{eq:key-les}
\cdots \rightarrow H^{*-2}(\bar{M})[[q]] \longrightarrow
H^*(\bar{E})_q \longrightarrow \mathit{HF}^*_{q,\bar{b}}(\bar\phi) \rightarrow \cdots
\end{equation}
Let's introduce an analogue of \eqref{eq:q-homology},
\begin{equation}
H^*(\cornerbar{E})_q = q^{-1}H^*(\cornerbar{E}, E|) \oplus H^*(\cornerbar{E})[[q]].
\end{equation}
By \eqref{eq:z1}, $z^{(1)}_{q,\cornersubbar{b}}$ admits a natural lift to $H^2(\cornerbar{E})_q$, for which we use the same notation.

\begin{lemma} \label{th:s-class}
The restriction of $z^{(1)}_{q,\cornersubbar{b}}$ to $H^2(\bar{E})_q$ equals the image of $1 \in H^0(\bar{M})[[q]]$ under the first map in \eqref{eq:key-les}.
\end{lemma}

Hence, the image of $\psi(q) \in H^0(\bar{M})[[q]]$ is $\psi(q) z^{(1)}_{q,\cornersubbar{b}}$. If one can achieve that this equals $q^{-1}[\delta E] + \partial_q\bar{b}/\bar{b}$, then it follows that the latter class maps to zero in $\mathit{HF}^*_{q,\bar{b}}(\bar\phi)$. This explains Theorem \ref{th:main}.

\begin{remark}
Suppose that 
\begin{equation} \label{eq:120}
H^2(E) = 0, \quad H^1_{\mathit{cpt}}(M) = 0.
\end{equation}
A more concrete form of the conjecture from \cite{seidel00b} says that the map
\begin{equation} \label{eq:open-closed-iso-0}
\mathit{HF}^*(\phi) \longrightarrow \mathit{HH}^*(\scrA,\scrA)
\end{equation}
obtained from \eqref{eq:0-diagram} is an isomorphism. Suppose that this is true. Then \eqref{eq:120} would imply that $\mathit{HH}^2(\scrA,\scrA) = 0$, by \eqref{eq:monodromy-les}. In that case, $\scrA$ would not have any nontrivial deformations at all. This fits in with our considerations as follows. The topological conditions \eqref{eq:120} imply that $H^2(\cornerbar{E})$ is two-dimensional, and it is then necessarily spanned by $[\bar{M}]$ and $[\delta E|]$. Assumption \ref{th:fundamental-assumption} is then trivially satisfied, for any $\cornerbar{b}$.
\end{remark}

\subsection{The main conjecture\label{subsec:conjecture}}
Our next step will be to discuss (conjecturally) how additional structures present in Fukaya categories of Lefschetz fibrations behave under fibrewise compactification. In \cite[Section 6]{seidel12b}, it was shown how to construct a canonical bimodule map
\begin{equation} \label{eq:serre-1}
\rho \in H^0(\mathit{hom}_{(\scrA,\scrA)}(\scrA^\vee[-n],\scrA)).
\end{equation}
Recall that the dual diagonal bimodule $\scrA^\vee$ gives rise (via tensor product) to the Serre functor on the category of right (perfect) $\scrA$-modules. Hence, \eqref{eq:serre-1} induces a natural transformation from the Serre functor (shifted up by $n$) to the identity functor. Given a fibrewise compactification, together with \eqref{eq:bulk}, one gets a corresponding deformation of \eqref{eq:serre-1} to an element
\begin{equation} \label{eq:deform-rho}
\rho_{q,\bar{b}} \in H^0(\mathit{hom}_{(\scrA_{q,\bar{b}},\scrA_{q,\bar{b}})}(\scrA^\vee_{q,\bar{b}}[-n],\scrA_{q,\bar{b}})).
\end{equation}

The next ingredient is more delicate. For Lefschetz fibrations coming from anticanonical Lefschetz pencils, \cite{seidel14b} provides a different construction of a map of the same kind, which we will here denote by
\begin{equation} \label{eq:serre-2}
\sigma \in H^0(\mathit{hom}_{(\scrA,\scrA)}(\scrA^\vee[-n],\scrA)).
\end{equation}
The construction in \cite{seidel14b} is rather ad hoc, but since our use of \eqref{eq:serre-2} is in any case speculative, we will allow ourselves to overlook its shortcomings. We conjecture that there is a canonical deformation of \eqref{eq:serre-2} which is formally parallel to \eqref{eq:deform-rho},
\begin{equation} \label{eq:deform-sigma}
\sigma_{q,\cornersubbar{b}} \in H^0(\mathit{hom}_{(\scrA_{q,\bar{b}},\scrA_{q,\bar{b}})}(\scrA^\vee_{q,\bar{b}}[-n],\scrA_{q,\bar{b}})).
\end{equation}
Suppose now that the bulk term is chosen as in Theorem \ref{th:main}, and hence that $\scrA_{q,\bar{b}}$ is the trivial deformation. After fixing a trivialization of that deformation, we get an isomorphism
\begin{equation} \label{eq:q-hom-serre}
H^*(\mathit{hom}_{(\scrA_{q,\bar{b}},\scrA_{q,\bar{b}})}(\scrA^\vee_{q,\bar{b}},\scrA_{q,\bar{b}})) \iso H^*(\mathit{hom}_{(\scrA,\scrA)}(\scrA^\vee,\scrA))[[q]].
\end{equation}
Hence, both $\rho_{q,\bar{b}}$ and $\sigma_{q,\cornersubbar{b}}$ can be considered as (formal) functions of $q$, taking values in $H^0(\mathit{hom}_{(\scrA,\scrA)}(\scrA^\vee[-n],\scrA))$. In particular, it now makes sense to differentiate them in $q$-direction.

\begin{conjecture} \label{th:connection}
Suppose that $\bar{b}$ is the restriction of some $\cornerbar{b}$ satisfying Assumption \ref{th:main}. Then, there is a choice of trivialization of the deformation $\scrA_{q,\bar{b}}$, such that, with respect to the associated isomorphism \eqref{eq:q-hom-serre}, we have
\begin{equation} \label{eq:connection-equation}
(\partial_q + \Gamma) 
\left(\begin{matrix} 
\rho_{q,\bar{b}} \\ 
\sigma_{q,\cornersubbar{b}} 
\end{matrix}
\right) = 0, \quad \text{
for the same $\Gamma$ as in \eqref{eq:fundamental-solution}.}
\end{equation}
\end{conjecture}

A usesful formal consistency check of \eqref{eq:connection-equation} is that it is compatible with the symmetries of \eqref{eq:fundamental-equation}. A change of $\cornerbar{b}$ as in \eqref{eq:symmetry-1} multiplies $z^{(2)}_{q,\cornersubbar{b}}$ with $\alpha^2$, and therefore transforms $\Gamma$ to
\begin{equation}
\Gamma_\alpha = \begin{pmatrix} 
0 & \alpha(q)^{-1} \psi(q) \\
\alpha(q) \, 4\psi(q)\, z^{(2)}_{q,\cornersubbar{b}} & \eta(q) - \alpha(q)^{-1} \alpha'(q)
\end{pmatrix}.
\end{equation}
The solutions of the corresponding equation \eqref{eq:connection-equation} are the given $\rho_{q,\bar{b}}$ (which indeed should not change, since $\bar{b}$ remains the same) together with
\begin{equation}
\sigma_{q,\cornersubbar{b}_\alpha}(q) = \alpha(q) \sigma_{q,\cornersubbar{b}}.
\end{equation}
Similarly, if we apply \eqref{eq:symmetry-2}, the corresponding transformation of $\Gamma$ is
\begin{equation}
\Gamma_\beta =
\beta'(q)\, \Gamma_{q,\cornersubbar{b}}(\beta(q))).
\end{equation}
The solutions of \eqref{eq:connection-equation} would then be obtained from the original ones by the same change of variables $q \mapsto \beta(q)$.

\begin{remark}
One of the more speculative aspects of Conjecture \ref{th:connection} is that it involves an object whose existence is itself conjectural, namely $\sigma_{q,\cornersubbar{b}}$. However, note that one can remove $\sigma_{q,\cornersubbar{b}}$ from the formalism, by rewriting \eqref{eq:connection-equation} as a second order linear ODE for $\rho_{q,\bar{b}}$ alone:
\begin{equation} \label{eq:second-order}
\partial_q^2 \rho_{q,\bar{b}} + \Big(\eta - \frac{\partial_q \psi}{\psi}\Big)\, \partial_q \rho_{q,\bar{b}} - \psi^2 4 z^{(2)}_{q,\cornersubbar{b}} \rho_{q,\bar{b}} = 0.
\end{equation}
In particular, if one knows the value of $\rho_{q,\bar{b}}$ at $q = 0$ and its first derivative, the entire $\rho_{q,\bar{b}}$ can be reconstructed from that using Gromov-Witten invariants of $\cornerbar{E}$.
\end{remark}

Part of the motivation for Conjecture \ref{th:eigenvalues} comes from the connection with holomorphic disc counts. Take $L$ as in \eqref{eq:disc-counting}, considered as an object of $\scrF_{q,\bar{b}}(\bar\pi)$. Since the restriction map
\begin{equation}
H^*(\mathit{hom}_{(\scrF_{q,\bar{b}}(\bar\pi),\scrF_{q,\bar{b}}(\bar\pi))}(\scrF_{q,\bar{b}}(\bar\pi)^\vee,\scrF_{q,\bar{b}}(\bar\pi))) \longrightarrow
H^*(\mathit{hom}_{(\scrA_{q,\bar{b}},\scrA_{q,\bar{b}})}(\scrA^\vee_{q,\bar{b}},\scrA_{q,\bar{b}}))
\end{equation}
is an isomorphism, the natural transformations \eqref{eq:deform-rho} and \eqref{eq:deform-sigma} extend uniquely to $\scrF_{q,\bar{b}}(\bar\pi)$. Moreover, the tensor product with the bimodule $\scrF_{q,\bar{b}}(\bar\pi)^\vee[-n]$ is the Serre functor for our category, which preserves closed Lagrangian submanifolds. Hence, we get elements
\begin{equation}
\rho_{L, q, \bar{b}}, \;\sigma_{L,q,\cornersubbar{b}} \in H^0(\mathit{hom}_{\scrF_{q,\bar{b}}(\bar\pi)}(L,L)) \iso \bC[[q]].
\end{equation}
One expects that
\begin{equation} \label{eq:wq}
W_{L,q,\cornersubbar{b}}= \frac{\sigma_{L,q,\cornersubbar{b}}}{\rho_{L,q,\bar{b}}}.
\end{equation}
This provides a link between Conjecture \ref{th:connection} and the previously observed behaviour of \eqref{eq:disc-counting} (see Lemma \ref{th:eigenvalues} and \eqref{eq:w-values}). Concretely, if $\rho$ is a nonzero solution of \eqref{eq:second-order}, and $\sigma = -(\partial_q\rho)/\psi$, then $W = -\sigma/\rho$ satisfies \eqref{eq:scalar-quadratic}.

\subsection{The Fukaya category of the fibre\label{subsec:fibre}}
Consider the category $\scrB$ from \eqref{eq:full}. The main result from \cite{seidel12} says that this has the structure of a noncommutative divisor on $\scrA$, where the relevant bimodule is $\scrP = \scrA^\vee[-n]$, and more importantly, the leading order term \eqref{eq:theta-class} can be identified with \eqref{eq:serre-1}. We would like to extend this idea to the fibrewise compactification, as follows:

\begin{conjecture} \label{th:get-fibre}
Suppose that the deformation $\scrA_{q,\bar{b}}$ is trivial. Then the category $\scrB_{q,\bar{b}|M}$ from \eqref{eq:bq-subcategory} has the structure of a deformation of noncommutative divisors on $\scrA$, whose leading order term \eqref{eq:theta-q-class} is the image of \eqref{eq:deform-rho} under \eqref{eq:q-hom-serre}.
\end{conjecture}

In general, the leading order term doesn't determine a noncommutative divisor. However, for Lefschetz fibrations arising from anticanonical Lefschetz pencils, it is reasonable to assume that this will be always the case, as explained in \cite[Remark 7.23]{seidel14b}. Assuming that this is the case, and that Conjectures \ref{th:connection} and \ref{th:get-fibre} hold, it follows that the map $\rho_{q,\bar{b}}$ computed through \eqref{eq:second-order} completely determines $\scrB_{q,\bar{b}|\bar{M}}$. In particular, we can then apply Corollary \ref{th:divide-out} with $t_1 = \Theta_{11}$ and $t_2 = \Theta_{12}$. This explains Conjecture \ref{th:3}, where the generator is in fact $t = \Theta_{12}/\Theta_{11}$.
%

\begin{remark} \label{th:mirror-fibre}
It is helpful to interpret this situation in terms of mirror symmetry. Suppose that $\scrA$ is derived equivalent to the category of coherent sheaves on a smooth variety $X$. The mirrors of $(\rho,\sigma)$ are $(r,s)$ sections of $\scrK_X^{-1}$. Suppose that this pencil gives a well-defined map
\begin{equation}
X \longrightarrow \bC P^1.
\end{equation}
Then, the mirror of $\scrB_{q,\bar{b}}$ would be obtained by looking at the vanishing locus of $\Theta_{11}r + \Theta_{12}s$, which means that it is the fibre over the formal disc in $\bC P^1$ parametrized by $-\Theta_{11}/\Theta_{12}$.
\end{remark}

\section{Elementary examples\label{sec:elementary}}

\subsection{A non-example}
Let's begin with an extremely simple ``made up'' example. This does not come from a Lefschetz pencil, hence our main results don't apply to it, but it illustrates the general idea of fibrewise compactification, and the associated deformation of the Fukaya category.

Consider an exact Lefschetz fibration whose fibre $M$ is a two-punctured torus, and which has a basis of vanishing cycles $(V_1,V_2)$ drawn in Figure \ref{fig:2-cycles} (this describes the Lefschetz fibration uniquely, up to deformation). We use the symplectic Calabi-Yau structure given by the standard trivialization of the canonical bundle of the torus. Using \eqref{eq:directed}, one gets
\begin{equation}
\mathit{hom}_{\scrA}(L_i,L_j) = \begin{cases} \bC \cdot e_{L_i} & i = j, \\
\mathit{CF}^*(V_1,V_2) = \bC \cdot x \oplus \bC \cdot y & (i,j) = (1,2), \\
0 & (i,j) = (2,1). \end{cases}
\end{equation}
The degrees are $|x| = 0$, $|y| = 1$. All $A_\infty$-operations $\mu^d_{\scrA}$ vanish. Let's consider the obvious fibrewise compactification, obtained by filling in the punctures, and use Lemma \ref{th:directed} to determine $\scrA_q$ (the same approach will be used in all the computations in this section). In this case, $\scrA_q$ is the trivial deformation: we have two holomorphic bigons, whose contributions to the differential $\mu^1_{\scrA_q}(x)$ are $\pm q y$, hence cancel. 

\begin{remark} \label{th:matching}
The partial compactification $\bar{E}$ contains a Lagrangian sphere, which gives rise to a spherical object in $\scrF_q(\bar\pi)$. Algebraically, this can be written as the cone of a morphism $L_1 \rightarrow L_2$, hence lies in the derived category $D(\scrA_q)$. Because the deformation is trivial, the same object appears already in $D(\scrA) \iso D(\scrF(\pi))$, in spite of the fact that there are no Lagrangian spheres in $E$ (since $H_2(E) = 0$).
\end{remark}

Let's consider the effect of introducing a bulk term \eqref{eq:bulk}. We can write it as
\begin{equation}
\bar{b} = (\beta_0)^{A_0} (\beta_1)^{A_1},
\end{equation}
where $A_0,A_1$ are the (Poincar{\'e} duals of the) two sections of $\bar{\pi}$ corresponding to the punctures in the fibre $M$; and $\beta_0$, $\beta_1$ are elements of $1 + q\bC[[q]]$. The effect on our previous computation is that holomorphic polygons in $\bar{M}$ that pass over the punctures with multiplicities $(m_0,m_1)$ should be counted with an additional weight $\beta_0^{m_0} \beta_1^{m_1}$. This means that
\begin{equation} \label{eq:beta-factors}
\mu^1_{\scrA_{q,\bar{b}}}(x) = q(\beta_0 - \beta_1) y.
\end{equation}
For $\beta_0 = \beta_1$, the deformation is trivial (and indeed, the general property \eqref{eq:change-of-bulk} tells us that $\scrF_{q,\bar{b}}(\bar\pi)$ must be a reparametrization of $\scrF_q(\bar\pi)$). For $\beta_0 \neq \beta_1$, the deformation is nontrivial (at some order in $q$, which depends on $\beta_1/\beta_0$). 

\begin{remark}
If $\beta_0 \neq \beta_1$, the Lagrangian sphere from Remark \ref{th:matching} is no longer an object of $\scrF_{q,\bar{b}}(\bar{\pi})$, since the integral of $\bar{b}$ over it is nonzero.
\end{remark}

\begin{remark} \label{th:add-puncture}
Suppose that we modify the geometric setup, by putting another puncture close to one of the two existing ones. This affects the way in which holomorphic bigons are counted, resulting in
\begin{equation} \label{eq:qq2}
\mu^1_{\scrA_q}(x) = (q-q^2) y,
\end{equation}
which is clearly a nontrivial deformation (even at first order in $q$). In this situation, there is no longer a Lagrangian sphere in $\bar{E}$, since the symplectic form will have nonzero integral over the relevant homology class. Moreover, $\scrA_{q,\bar{b}}$ will remain a nontrivial deformation for any choice of $\bar{b}$, since that changes the weights with which holomorphic bigons are counted only by invertible elements of $\bC[[q]]$. Thus, Question \ref{qu:1} has a negative answer in this case.
\end{remark}

\begin{figure}
\begin{centering}
\begin{picture}(0,0)%
\includegraphics{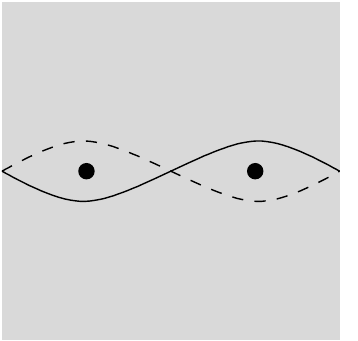}%
\end{picture}%
\setlength{\unitlength}{3552sp}%
\begingroup\makeatletter\ifx\SetFigFont\undefined%
\gdef\SetFigFont#1#2#3#4#5{%
  \reset@font\fontsize{#1}{#2pt}%
  \fontfamily{#3}\fontseries{#4}\fontshape{#5}%
  \selectfont}%
\fi\endgroup%
\begin{picture}(1902,1802)(589,-1412)
\put(2476,-561){\makebox(0,0)[lb]{\smash{{\SetFigFont{10}{12}{\rmdefault}{\mddefault}{\updefault}{\color[rgb]{0,0,0}$y$}%
}}}}
\put(1926,-886){\makebox(0,0)[lb]{\smash{{\SetFigFont{10}{12}{\rmdefault}{\mddefault}{\updefault}{\color[rgb]{0,0,0}$2$}%
}}}}
\put(1926,-286){\makebox(0,0)[lb]{\smash{{\SetFigFont{10}{12}{\rmdefault}{\mddefault}{\updefault}{\color[rgb]{0,0,0}$1$}%
}}}}
\put(1501,-411){\makebox(0,0)[lb]{\smash{{\SetFigFont{10}{12}{\rmdefault}{\mddefault}{\updefault}{\color[rgb]{0,0,0}$x$}%
}}}}
\end{picture}%
\caption{\label{fig:2-cycles}}
\end{centering}
\end{figure}%

\subsection{The mirror of $\bC P^2$} 
Our next example is the toric mirror of the projective plane: 
\begin{equation} \label{eq:toric-mirror-cp2}
\begin{aligned}
& E = \bC^* \times \bC^*, \\
& \pi(x_1,x_2) = x_1 + x_2 + \frac{1}{x_1x_2}.
\end{aligned}
\end{equation}
We'll use the basis of vanishing cycles from \cite[Section 3b]{seidel00b}, shown in Figure \ref{fig:cp2-mirror}. This is the example we had in mind in Section \ref{subsec:the-problem}, for which it was shown in \cite[Section 6.2]{auroux-katzarkov-orlov04} that the deformation $\scrA_q$ is trivial. Within mirror symmetry, this reflects the fact that the fibrewise compactification is again a mirror of $\bC P^2$, but with respect to a non-toric anticanonical divisor \cite{auroux07} (we will return to this viewpoint later).
\begin{figure}
\begin{centering}
\begin{picture}(0,0)%
\includegraphics{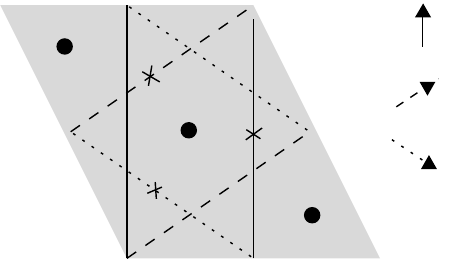}%
\end{picture}%
\setlength{\unitlength}{3552sp}%
\begingroup\makeatletter\ifx\SetFigFont\undefined%
\gdef\SetFigFont#1#2#3#4#5{%
  \reset@font\fontsize{#1}{#2pt}%
  \fontfamily{#3}\fontseries{#4}\fontshape{#5}%
  \selectfont}%
\fi\endgroup%
\begin{picture}(2416,1376)(1200,-1123)
\put(3601, 14){\makebox(0,0)[lb]{\smash{{\SetFigFont{10}{12}{\rmdefault}{\mddefault}{\updefault}{\color[rgb]{0,0,0}$1$}%
}}}}
\put(3601,-586){\makebox(0,0)[lb]{\smash{{\SetFigFont{10}{12}{\rmdefault}{\mddefault}{\updefault}{\color[rgb]{0,0,0}$3$}%
}}}}
\put(3601,-286){\makebox(0,0)[lb]{\smash{{\SetFigFont{10}{12}{\rmdefault}{\mddefault}{\updefault}{\color[rgb]{0,0,0}$2$}%
}}}}
\end{picture}%
\caption{\label{fig:cp2-mirror}}
\end{centering}
\end{figure}%

Let's start with the exact fibration (except for signs, this computation goes back to \cite{seidel00b}). Since there are three vanishing cycles, and all the morphisms between them are concentrated in degree $0$, the only nontrivial part of the $A_\infty$-structure is the product
\begin{equation} \label{eq:cp2-product}
\mu^2_{\scrA}: \mathit{CF}^*(V_2,V_3) \otimes \mathit{CF}^*(V_1,V_2) \longrightarrow \mathit{CF}^*(V_1,V_3).
\end{equation}
The intersections $V_i \cap V_j$ ($i \neq j$) each consist of three points, so \eqref{eq:cp2-product} has $27$ coefficients. All but $6$ of these coefficients vanish, and the remaining ones are $\pm 1$ (corresponding to the six ``small'' triangles in Figure \ref{fig:cp2-mirror}). To get the signs right, it is important to remember that all vanishing cycles must be equipped with local systems having holonomy $-1$, or equivalently, with nontrivial {\em Spin} structures (the reasons are explained in \cite{seidel04}). Let's think of these local systems (or {\em Spin} structures) as being trivial away from one marked point on each $V_k$; the specific choice of points we use is shown in Figure \ref{fig:cp2-mirror}. The effect is that a triangle is counted with $(-1)^s$ if its boundary crosses the marked points $s$ times. With respect to bases $\{x_k\}$, $\{y_k\}$, $\{z_k\}$ of the spaces in \eqref{eq:cp2-product} given by intersection points, one then has
\begin{equation} \label{eq:nc-relations-0}
\begin{aligned}
& \mu^2_{\scrA}(y_2,x_1) = z_3, &
& \mu^2_{\scrA}(y_1,x_2) = -z_3, &
& \mu^2_{\scrA}(y_3,x_2) = z_1, \\
& \mu^2_{\scrA}(y_2,x_3) = -z_1, &
& \mu^2_{\scrA}(y_1,x_3) = z_2, &
& \mu^2_{\scrA}(y_3,x_1) = -z_2, \\
& \mu^2_{\scrA}(y_1,x_1) = 0, &
& \mu^2_{\scrA}(y_2,x_2) = 0, &
& \mu^3_{\scrA}(y_3,x_3) = 0.
\end{aligned}
\end{equation}
As explained in \cite{auroux-katzarkov-orlov04}, the corresponding products $\mu^2_{\scrA_q}$ are obtained from those in \eqref{eq:nc-relations-0} by multiplying with
\begin{equation} \label{eq:rho-function}
\gamma(q) = \sum_{j=0}^{\infty} (-1)^j q^{3j(3j+1)/2} - (-1)^j q^{(3j+2)(3j+3)/2} = 1-q^3-q^6 + \cdots
\end{equation}
Since the relations between these products remain the same, the deformation is trivial (in fact, it can be trivialized by multiplying all $z_k$ with $\gamma$). Remarkably, the fact that $\mu^2_{\scrA_q}(y_k,x_k)$ vanishes does not come from the absence of holomorphic triangles, but rather from an infinite number of cancellations between their contributions, at all orders in $q$.

\begin{remark} \label{th:rigid}
There is a way of explaining the triviality of $\scrA_q$ a priori, without doing any specific computations, by looking at the deformation theory of the derived ($A_\infty$ or dg) categories
\begin{equation} \label{eq:cp2-mirror}
D^b\scrF(\pi) \iso D^b\mathit{Coh}(\bC P^2).
\end{equation}

Recall that the total space of our Lefschetz fibration is $E = (\bC^*)^2$. This contains an exact Lagrangian torus $T \subset E$, which becomes an object of $\scrF(\pi)$. Let $\scrT \subset \scrF(\pi)$ be the full $A_\infty$-subcategory consisting of only that object ($\scrT$ is formal, hence an exterior algebra in $2$ variables). Any holomorphic disc in $\bar{E}$ with boundary on $T$ must have Maslov index $0$, since the relative first Chern class vanishes. For dimension reasons, there are (generically) no such discs, hence $\scrF_q(\bar\pi)$ restricts to the trivial deformation of $\scrT$. 

The mirror of $T$ is the structure sheaf $P$ of a point in the open torus orbit of $\bC P^2$. Let $\scrP \subset D^b\mathit{Coh}(\bC P^2)$ be the full subcategory consisting only of that structure sheaf. By the Hochschild-Kostant-Rosenberg theorem, one can identify the restriction map on Hochschild cohomology,
\begin{equation}
\mathit{HH}^2(\bC P^2) \longrightarrow \mathit{HH}^2(\scrP,\scrP),
\end{equation}
with the map that takes a bi-vector field (a section of the second exterior power of the tangent bundle) to its Taylor expansion near $P$. This map is clearly injective. From the derived equivalence \eqref{eq:cp2-mirror} one sees that the corresponding restriction map
\begin{equation} \label{eq:restriction-to-formal}
\mathit{HH}^2(\scrF(\pi),\scrF(\pi)) \longrightarrow \mathit{HH}^2(\scrT,\scrT)
\end{equation}
is also injective. A formal deformation of $\scrF(\pi)$ which is nontrivial at some order $q^d$ determines a nonzero class in $\mathit{HH}^2(\scrF(\pi),\scrF(\pi))$, hence (by the injectivity of \eqref{eq:restriction-to-formal}, and some general considerations concerning deformations of $\scrT$) restricts to a nontrivial deformation of $\scrT$. Because of our previous observation concerning $\scrF_q(\bar\pi)$, that deformation must then be trivial.

The same argument applies to the mirrors of the other smooth toric Fano surfaces \cite{auroux-katzarkov-orlov04, auroux-katzarkov-orlov05, ueda06}. In higher dimensions, it is still true that smooth toric Fano varieties are rigid \cite{bien-brion96}, hence all their deformations are noncommutative ones. In principle, there are methods for proving analogues of our statement concerning Maslov index $0$ holomorphic discs \cite{fukaya-oh-ohta-ono10}. However, the geometry of the natural (toric) fibrewise compactification is much more complicated.
\end{remark}

It is also instructive to consider deformations with bulk term. Let's assume that $\bar{b}$ comes from a class in 
\begin{equation} \label{eq:h2}
H^2(\bar{E},\bar{M}; 1 + q\bC[[q]]) \iso (1+q\bC^3[[q]])^3
\end{equation}
(here, $\bar{M}$ is placed ``at infinity''). An explicit choice of isomorphism \eqref{eq:h2} is given by integrating a cohomology class over the Lefschetz thimbles, which yields a triple $(\beta_1,\beta_2,\beta_3)$ of elements of $1+q\bC[[q]]$. In the same terms, the boundary operator $H^1(\bar{M}) \rightarrow H^2(\bar{E},\bar{M})$ is given by integrating one-forms on $\bar{M}$ over vanishing cycles, hence its image is spanned by $(1,2,1)$ and $(2,1,-1)$. This means that $\bar{b} \in H^2(\bar{E}; 1+q\bC[[q]])$ is determined by
\begin{equation} \label{eq:beta-difference}
\beta_2 / (\beta_1 \beta_3) \in 1 + q\bC[[q]].
\end{equation}

When carrying out the computation of $\scrA_{q,\bar{b}}$ as the directed subcategory of $\scrB_{q,\bar{b}|\bar{M}}$, the vanishing cycle $V_k$ should be equipped with a local system having holonomy $-\beta_k$, rather than the original $-1$ (as before, we think of the local system as being trivial away from the marked points). This leads to the following deformation of \eqref{eq:nc-relations-0}:
\begin{equation} \label{eq:nc-relations}
\begin{aligned}
& \mu^2_{\scrA_{q,\bar{b}}}(y_2,x_1) = z_3 + O(q^2), &
& \mu^2_{\scrA_{q,\bar{b}}}(y_1,x_2) = - \beta_3^{-1} z_3 + O(q^2), \\
& \mu^2_{\scrA_{q,\bar{b}}}(y_3,x_2) = z_1 + O(q^2), &
& \mu^2_{\scrA_{q,\bar{b}}}(y_2,x_3) = - \beta_1^{-1} z_1 + O(q^2), \\
& \mu^2_{\scrA_{q,\bar{b}}}(y_1,x_3) = z_2 + O(q^2), &
& \mu^2_{\scrA_{q,\bar{b}}}(y_3,x_1) = - \beta_2 z_2 + O(q^2), \\
& \mu^2_{\scrA_{q,\bar{b}}}(y_1,x_1) = O(q^2), &
& \mu^2_{\scrA_{q,\bar{b}}}(y_2,x_2) = O(q^2), \\
& \mu^3_{\scrA_{q,\bar{b}}}(y_3,x_3) = O(q^2).
\end{aligned}
\end{equation}
Consider solutions of the Maurer-Cartan equation
\begin{equation}
\mu^2_{\scrA_{q,\bar{b}}}(\eta_1y_1+\eta_2y_3+\eta_3y_3, \xi_1 x_1 + \xi_2x_2 + \xi_3x_3) = 0,
\end{equation}
with coefficients $\eta_k,\xi_k$ which are functions of $q$, and which have nonzero constant terms. From \eqref{eq:nc-relations}, one gets
\begin{equation} \label{eq:mc-relations}
\begin{aligned}
& \textstyle \frac{\eta_2}{\eta_1} = \beta_3^{-1} \frac{\xi_2}{\xi_1} + O(q^2), &
& \textstyle \frac{\eta_3}{\eta_2} = \beta_1^{-1} \frac{\xi_3}{\xi_2} + O(q^2), &
& \textstyle \frac{\eta_1}{\eta_3} = \beta_2 \frac{\xi_1}{\xi_3} + O(q^2).
\end{aligned}
\end{equation}
If $\bar{b}$ is nontrivial to first order, which means that \eqref{eq:beta-difference} does not lie in $1+O(q^2)$, the equations \eqref{eq:mc-relations} are not solvable. Since solutions of Maurer-Cartan equations transfer across isomorphisms, it then follows that $\scrA_{q,\bar{b}}$ is nontrivial.

\subsection{The anticanonical pencil on the cubic surface\label{subsec:cubic-pencil}}
The previous example is still not a Lefschetz pencil, but one can enlarge it to one, by adding $9$ more vanishing cycles: more precisely, three copies of each of the three disjoint curves shown in Figure \ref{fig:add9} (the ordering within that group of $9$ is irrelevant). The resulting basis of $12$ vanishing cycles is what one gets from the anticanonical Lefschetz pencil on a cubic surface (the projective plane blown up at $6$ general points). 

It turns out that the deformation $\scrA_q$ is still trivial (which of course is strictly stronger than our previous observation). Instead of addressing this by direct computation, we will give a mirror symmetry explanation, which is not quite rigorous (it is a degenerate version of  arguments from \cite{auroux-katzarkov-orlov05}).
\begin{figure}
\begin{centering}
\begin{picture}(0,0)%
\includegraphics{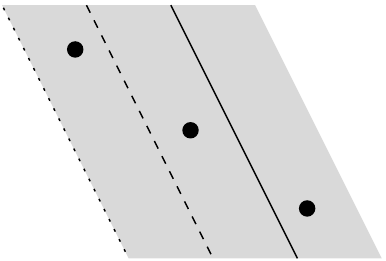}%
\end{picture}%
\setlength{\unitlength}{3552sp}%
\begingroup\makeatletter\ifx\SetFigFont\undefined%
\gdef\SetFigFont#1#2#3#4#5{%
  \reset@font\fontsize{#1}{#2pt}%
  \fontfamily{#3}\fontseries{#4}\fontshape{#5}%
  \selectfont}%
\fi\endgroup%
\begin{picture}(2038,1374)(1189,-1123)
\end{picture}%
\caption{\label{fig:add9}}
\end{centering}
\end{figure}%

For a complex number $q$, $0<|q|<1$, consider the following elliptic curve $Y_q$ together with a degree $3$ divisor $Z_q$:
\begin{align}
& Y_q = \bC/(\bZ \oplus \textstyle\frac{\log(q)}{2\pi i} \bZ), \\
& Z_q = \{0,1/3,2/3\} = (\textstyle\frac{1}{3}\bZ)/\bZ \subset Y_q. \label{eq:z-divisor}
\end{align}
The associated line bundle $\scrO_{Y_q}(Z_q)$ gives rise to an embedding
\begin{equation} \label{eq:projective-embedding}
Y_q \hookrightarrow \bP(H^0(Y_q,\scrO_{Y_q}(Z_q))^\vee) \iso \bC P^2. 
\end{equation}
Blow up the image of $Z_q$ under that embedding, and lift the embedding to the blowup. Then, carry out the same process two more times. The outcome is a rational elliptic surface $X_q$, into which $Y_q$ is embedded with selfintersection $0$. The derived category of $X_q$ carries a full exceptional collection of $12$ objects \cite{bondal-orlov95}, whose restriction to $Y_q$ is mirror to our basis of vanishing cycles on the (compact) torus with symplectic area $-\log(q)$. In particular, if we temporarily pretend that $q$ can be treated as a complex number on the symplectic geometry side, $\scrF_q(\bar\pi)$ would be derived equivalent to the category of coherent sheaves on $X_q$.
\begin{figure}
\begin{centering}
\begin{picture}(0,0)%
\includegraphics{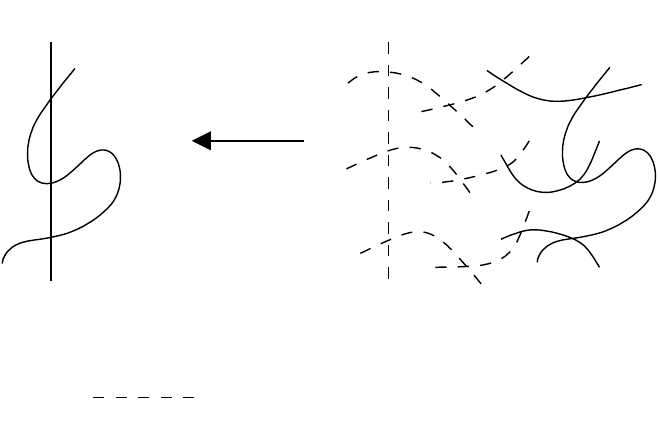}%
\end{picture}%
\setlength{\unitlength}{3552sp}%
\begingroup\makeatletter\ifx\SetFigFont\undefined%
\gdef\SetFigFont#1#2#3#4#5{%
  \reset@font\fontsize{#1}{#2pt}%
  \fontfamily{#3}\fontseries{#4}\fontshape{#5}%
  \selectfont}%
\fi\endgroup%
\begin{picture}(3512,2236)(1529,-1325)
\put(2026,-1036){\makebox(0,0)[lb]{\smash{{\SetFigFont{10}{12}{\rmdefault}{\mddefault}{\updefault}{\color[rgb]{0,0,0}numbers are selfintersections,}%
}}}}
\put(1726,764){\makebox(0,0)[lb]{\smash{{\SetFigFont{10}{12}{\rmdefault}{\mddefault}{\updefault}{\color[rgb]{0,0,0}$1$}%
}}}}
\put(5026,-286){\makebox(0,0)[lb]{\smash{{\SetFigFont{10}{12}{\rmdefault}{\mddefault}{\updefault}{\color[rgb]{0,0,0}$0$}%
}}}}
\put(2176,-286){\makebox(0,0)[lb]{\smash{{\SetFigFont{10}{12}{\rmdefault}{\mddefault}{\updefault}{\color[rgb]{0,0,0}$9$}%
}}}}
\put(4726,-586){\makebox(0,0)[lb]{\smash{{\SetFigFont{10}{12}{\rmdefault}{\mddefault}{\updefault}{\color[rgb]{0,0,0}$-1$}%
}}}}
\put(4726,164){\makebox(0,0)[lb]{\smash{{\SetFigFont{10}{12}{\rmdefault}{\mddefault}{\updefault}{\color[rgb]{0,0,0}$-1$}%
}}}}
\put(4951,464){\makebox(0,0)[lb]{\smash{{\SetFigFont{10}{12}{\rmdefault}{\mddefault}{\updefault}{\color[rgb]{0,0,0}$-1$}%
}}}}
\put(2626,-1261){\makebox(0,0)[lb]{\smash{{\SetFigFont{10}{12}{\rmdefault}{\mddefault}{\updefault}{\color[rgb]{0,0,0}= selfintersection $-2$}%
}}}}
\end{picture}%
\caption{\label{fig:blowups}}
\end{centering}
\end{figure}

The key observation is that $X = X_q$ is actually independent of $q$. By definition of \eqref{eq:projective-embedding}, the image of $Z_q$ consists of three collinear points in $\bC P^2$ (the left side of Figure \ref{fig:blowups} shows the image of $Y_q$, and the line through the image of $Z_q$; the right side shows what happens after repeatedly blowing up). Specifically for our choice \eqref{eq:z-divisor}, we have linear equivalences
\begin{equation}
Z_q \sim 3 \cdot \{0\} \sim 3 \cdot \{1/3\} \sim 3 \cdot \{2/3\}.
\end{equation}
This means that for each point of $Z_q$, there is a section of $\scrO_{Y_q}(Z_q)$ which vanishes to third order at that point. If we use those sections as a basis of our linear system, the image of $Y_q$ under \eqref{eq:projective-embedding} will intersect the coordinate lines at the points of $Z_q$, and each time with intersection multiplicity $3$. As a consequence, the repeated blowups used to form $X_q$ are carried out on the proper transforms of the coordinate lines, hence are the same for all $q$.

\begin{remark}
To make this argument into a rigorous proof of the triviality of $\scrA_q$, one would need to modify it in two ways. First, $q$ needs to be treated as a formal variable, rather than as a complex number. Second, one has to extend it to $q = 0$, where the elliptic curve degenerates to the union of coordinate lines (the corresponding version of homological mirror symmetry is well-known \cite{sibilla11, lekili-perutz11}).
\end{remark}

\begin{remark} \label{th:auroux}
As one sees from Figure \ref{fig:blowups}, $X$ contains an $\tilde{E}_6$ configuration of $(-2)$-spheres. The mirror is a configuration of $8$ Lagrangian spheres in the affine cubic surface, which is a basis of vanishing cycles for the $T_{3,3,3}$ singularity \cite[Section 4.2]{keating14}. $X$ also contains another $\tilde{A}_2$ configuration, disjoint from the previous one, namely the proper transforms of the coordinate lines (if there are mirror spheres to those, they need to lie in $\bar{E}$, which would give rise to a phenomenon similar to that in Remark \ref{th:matching}). The existence of those two singular elliptic fibres characterizes $X$ uniquely. It is an extremal elliptic surface \cite{persson90}, $X = X_{431}$ in the notation from \cite[p.~77]{miranda89}.
\end{remark}

\subsection{The anticanonical pencil on $\bC P^2$\label{subsec:cp2-pencil}}
Figure \ref{fig:12} shows a basis of vanishing cycles for the anticanonical (degree $3$) pencil on $\bC P^2$, where $M$ is the nine-punctured torus. The cycles are labeled as $V_{k,t}$ for $k \in \{1,2,3,4\}$ and $t \in \{\frac{1}{2},\frac{3}{2},\frac{5}{2}\}$, and should be ordered nondecreasingly in $k$ (within each group with fixed $k$, the ordering is irrelevant). The remark about filling in Figure \ref{fig:12}, and the shaded triangles, should be ignored for now.

In this case, the deformation $\scrA_q$ is once more trivial. One could give a mirror symmetry argument along the lines of the previous one (the mirror is the extremal elliptic surface $X_{9111}$). However, we will instead proceed by direct computation, which is easier this time since no two vanishing cycles are isotopic.

\begin{figure}
\begin{centering}
\begin{picture}(0,0)%
\includegraphics{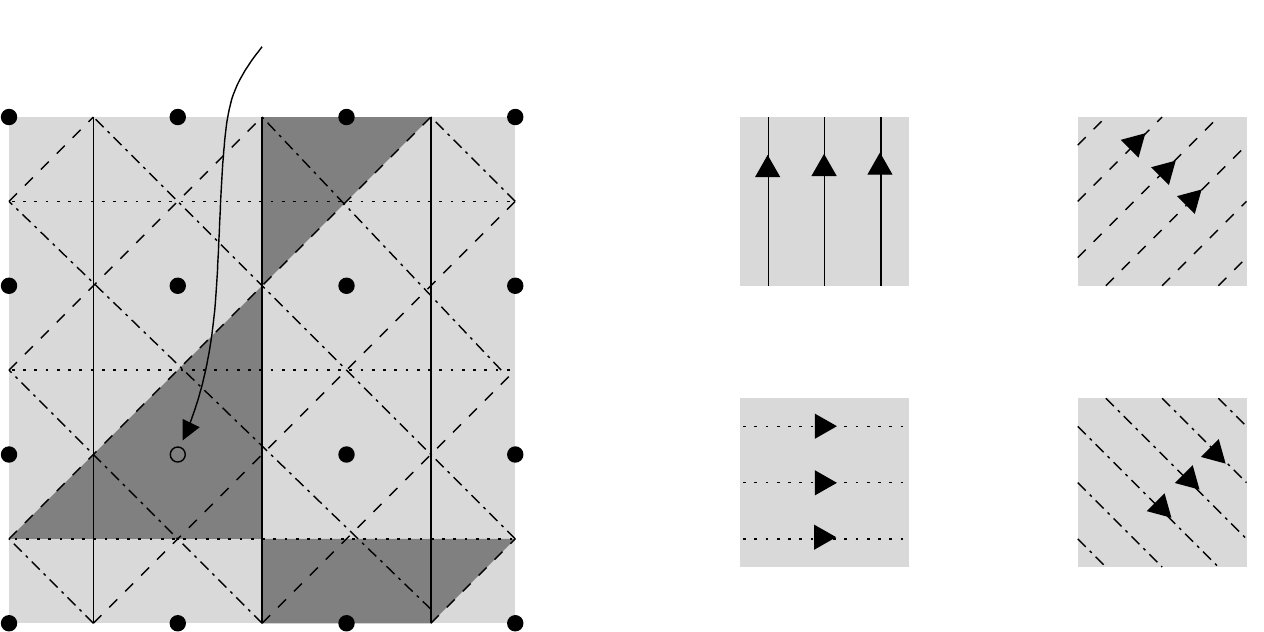}%
\end{picture}%
\setlength{\unitlength}{3552sp}%
\begingroup\makeatletter\ifx\SetFigFont\undefined%
\gdef\SetFigFont#1#2#3#4#5{%
  \reset@font\fontsize{#1}{#2pt}%
  \fontfamily{#3}\fontseries{#4}\fontshape{#5}%
  \selectfont}%
\fi\endgroup%
\begin{picture}(6738,3356)(-1997,-1908)
\put(-749,1289){\makebox(0,0)[lb]{\smash{{\SetFigFont{10}{12}{\rmdefault}{\mddefault}{\updefault}{\color[rgb]{0,0,0}(to be filled in)}%
}}}}
\put(2251,914){\makebox(0,0)[lb]{\smash{{\SetFigFont{9}{10.8}{\rmdefault}{\mddefault}{\updefault}{\color[rgb]{0,0,0}$(1,\frac32)$}%
}}}}
\put(1951,-241){\makebox(0,0)[lb]{\smash{{\SetFigFont{9}{10.8}{\rmdefault}{\mddefault}{\updefault}{\color[rgb]{0,0,0}$(1,\half)$}%
}}}}
\put(2551,-241){\makebox(0,0)[lb]{\smash{{\SetFigFont{9}{10.8}{\rmdefault}{\mddefault}{\updefault}{\color[rgb]{0,0,0}$(1,\frac 52)$}%
}}}}
\put(2926,-860){\makebox(0,0)[lb]{\smash{{\SetFigFont{9}{10.8}{\rmdefault}{\mddefault}{\updefault}{\color[rgb]{0,0,0}$(3,\frac52)$}%
}}}}
\put(2930,-1167){\makebox(0,0)[lb]{\smash{{\SetFigFont{9}{10.8}{\rmdefault}{\mddefault}{\updefault}{\color[rgb]{0,0,0}$(3,\frac32)$}%
}}}}
\put(2926,-1460){\makebox(0,0)[lb]{\smash{{\SetFigFont{9}{10.8}{\rmdefault}{\mddefault}{\updefault}{\color[rgb]{0,0,0}$(3,\frac12)$}%
}}}}
\put(4726,-1186){\makebox(0,0)[lb]{\smash{{\SetFigFont{9}{10.8}{\rmdefault}{\mddefault}{\updefault}{\color[rgb]{0,0,0}$(4,\frac32)$}%
}}}}
\put(4726,-886){\makebox(0,0)[lb]{\smash{{\SetFigFont{9}{10.8}{\rmdefault}{\mddefault}{\updefault}{\color[rgb]{0,0,0}$(4,\frac52)$}%
}}}}
\put(4726,-1486){\makebox(0,0)[lb]{\smash{{\SetFigFont{9}{10.8}{\rmdefault}{\mddefault}{\updefault}{\color[rgb]{0,0,0}$(4,\frac12)$}%
}}}}
\put(4718,640){\makebox(0,0)[lb]{\smash{{\SetFigFont{9}{10.8}{\rmdefault}{\mddefault}{\updefault}{\color[rgb]{0,0,0}$(2,\frac52)$}%
}}}}
\put(4718, 40){\makebox(0,0)[lb]{\smash{{\SetFigFont{9}{10.8}{\rmdefault}{\mddefault}{\updefault}{\color[rgb]{0,0,0}$(2,\frac12)$}%
}}}}
\put(4717,341){\makebox(0,0)[lb]{\smash{{\SetFigFont{9}{10.8}{\rmdefault}{\mddefault}{\updefault}{\color[rgb]{0,0,0}$(2,\frac32)$}%
}}}}
\end{picture}%
\caption{\label{fig:12}}
\end{centering}
\end{figure}%
\begin{figure}
\begin{centering}
\begin{picture}(0,0)%
\includegraphics{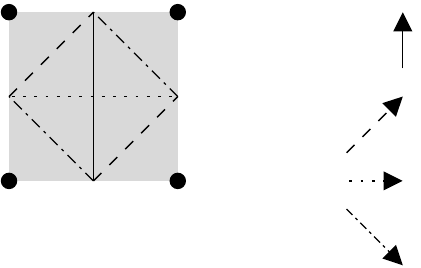}%
\end{picture}%
\setlength{\unitlength}{3552sp}%
\begingroup\makeatletter\ifx\SetFigFont\undefined%
\gdef\SetFigFont#1#2#3#4#5{%
  \reset@font\fontsize{#1}{#2pt}%
  \fontfamily{#3}\fontseries{#4}\fontshape{#5}%
  \selectfont}%
\fi\endgroup%
\begin{picture}(2238,1409)(-1997,-523)
\put(226,-136){\makebox(0,0)[lb]{\smash{{\SetFigFont{10}{12}{\rmdefault}{\mddefault}{\updefault}{\color[rgb]{0,0,0}$3$}%
}}}}
\put(226,164){\makebox(0,0)[lb]{\smash{{\SetFigFont{10}{12}{\rmdefault}{\mddefault}{\updefault}{\color[rgb]{0,0,0}$2$}%
}}}}
\put(226,614){\makebox(0,0)[lb]{\smash{{\SetFigFont{10}{12}{\rmdefault}{\mddefault}{\updefault}{\color[rgb]{0,0,0}$1$}%
}}}}
\put(226,-436){\makebox(0,0)[lb]{\smash{{\SetFigFont{10}{12}{\rmdefault}{\mddefault}{\updefault}{\color[rgb]{0,0,0}$4$}%
}}}}
\end{picture}%
\caption{\label{fig:12-quotient}}
\end{centering}
\end{figure}

\begin{remark} \label{th:z3-action}
There is a formal relationship between this example and that in Section \ref{subsec:cubic-pencil}. Returning to Figure \ref{fig:cp2-mirror}, note that the three curves drawn there span an index $3$ subgroup of $H_1(\bar{M}) = \bZ^2$. Hence, there is a unique $\bZ/3$-cover of $\bar{M}$ which is trivial when restricted to each curve. That cover is nontrivial on the curves from Figure \ref{fig:add9}. On that cover, the $9$ preimages of the curves from Figure \ref{fig:cp2-mirror}, together with the preimage of each of the $3$ curves in Figure \ref{fig:add9} (without repetition), form the collection shown in Figure \ref{fig:12} (the lack of immediate visual resemblance comes from drawing the torus differently). On the level of $A_\infty$-structures $\scrA_q$, this means that the one currently under consideration is obtained from the previous one by taking the semidirect product with respect to an action of $\bZ/3$ (hence, triviality of the deformation in one case can in fact be derived from that in the other one). Dually, the mirror $X_{9111}$ carries a $\bZ/3$-action whose resolved quotient is the previously considered $X_{431}$.
\end{remark}

Consider first the simpler situation from Figure \ref{fig:12-quotient}: a collection of four curves $\{V_k\}$ on a once-punctured torus. Each curve should be equipped with a line bundle whose holonomy is $\zeta_k$, where 
\begin{equation} \label{eq:minus-root}
\zeta_k^3 = -1. 
\end{equation}
This time, we want to think of it as the trivial line bundle with a constant connection given by a choice of logarithm $u_k = \log(\zeta_k)/2\pi i \in \frac16 + \frac13 \bZ$ (and use that logarithm to determine roots of $\zeta_k$, which will appear in the formulae below). Let's denote the resulting objects by $V_{k,\zeta_k}$. When counting holomorphic polygons, the effect is that a side of a triangle which covers a fraction $\phi$ of the curve $V_{k,\zeta_k}$ will contribute a holonomy term $e^{2\pi i u_k \phi}$.

\begin{prerequisites}
Two of the Jacobi theta-functions are
\begin{equation}
\begin{aligned}
& \theta_2(u,q) = \sum_{d \in \bZ + \half} e^{2\pi i ud} q^{d^2} =
(e^{\pi i u} + e^{-\pi i u})q^{1/4} + O(q^{9/4}), \\
& \theta_3(u,q) = \sum_{d \in \bZ} e^{2\pi i u d} q^{d^2} =
1 + O(q).
\end{aligned}
\end{equation}
We will need the obvious periodicity properties
\begin{equation} \label{eq:theta-symmetry}
\begin{aligned}
& \theta_2(u+1,q) = -\theta_2(u,q), \\
& \theta_3(u+1,q) = \theta_3(u,q)
\end{aligned}
\end{equation}
as well as Watson's identity (see e.g.\ \cite[p.~130--131]{mckean-moll99})
\begin{equation} \label{eq:watson}
\theta_3(u+v,q)\theta_2(u-v,q) + \theta_2(u+v,q)\theta_3(u-v,q) = \theta_2(u,q^{1/2})\theta_2(v,q^{1/2}).
\end{equation}
Finally, note the following special value:
\begin{equation} \label{eq:special-value}
\theta_2(u,q^{1/2}) = (e^{\pi i u} + e^{-\pi i u}) q^{1/8}\gamma(q) \quad
\text{if $\textstyle u \in \frac{1}{6} + \frac{1}{3}\bZ$, with $\gamma$ as in \eqref{eq:rho-function}}.
\end{equation}
\end{prerequisites}

In Figure \ref{fig:12-quotient}, let's consider the Floer-theoretic products
\begin{align} 
& \label{eq:123-product}
\mathit{CF}^*(V_{2,\zeta_2}, V_{3,\zeta_3}) \otimes \mathit{CF}^*(V_{1,\zeta_1},V_{2,\zeta_2}) \iso \bC \longrightarrow \mathit{CF}^*(V_{1,\zeta_1}, V_{3,\zeta_3}) \iso \bC, \\
& \label{eq:134-product}
\mathit{CF}^*(V_{3,\zeta_3}, V_{4,\zeta_4}) \otimes \mathit{CF}^*(V_{1,\zeta_1},V_{3,\zeta_3}) \iso \bC \longrightarrow \mathit{CF}^*(V_{1,\zeta_1}, V_{4,\zeta_4}) \iso \bC, \\
& \label{eq:124-product}
\mathit{CF}^*(V_{2,\zeta_2}, V_{4,\zeta_4}) \otimes \mathit{CF}^*(V_{1,\zeta_1},V_{2,\zeta_2}) \iso \bC^2 \longrightarrow \mathit{CF}^*(V_{1,\zeta_1}, V_{4,\zeta_4}) \iso \bC, \\
& \label{eq:234-product}
\mathit{CF}^*(V_{3,\zeta_3}, V_{4,\zeta_4}) \otimes \mathit{CF}^*(V_{2,\zeta_2},V_{3,\zeta_3}) \iso \bC \longrightarrow \mathit{CF}^*(V_{2,\zeta_2}, V_{4,\zeta_4}) \iso \bC^2
\end{align}
and their $q$-deformations. These can be determined combinatorially by triangle-counting. In the natural basis given by the intersection points, the outcome is given by, respectively,
\begin{align} 
&
q^{-1/8} \theta_2(u_1 - u_2 + u_3, q^{1/2}) = (\zeta_1^{1/2}\zeta_2^{-1/2}\zeta_3^{1/2} + \zeta_1^{-1/2}\zeta_2^{1/2}\zeta_3^{-1/2}) \gamma(q) \\
&
q^{-1/8} \theta_2(u_1 - u_3 + u_4, q^{1/2}) = (\zeta_1^{1/2}\zeta_3^{-1/2}\zeta_4^{1/2} + \zeta_1^{-1/2}\zeta_3^{1/2}\zeta_4^{-1/2}) \gamma(q), \\
&
\big(\theta_3(2u_1 - u_2 + u_4,q), q^{-1/4} \theta_2(2u_1 - u_2 + u_4, q)\big), \\
&
\begin{pmatrix} 
q^{-1/4} \theta_2(u_2 - 2u_3 + u_4,q) \\
\theta_3(u_2 - 2u_3 + u_4,q)
\end{pmatrix}
\end{align}
(we have used \eqref{eq:special-value} in the first two equations). Again using \eqref{eq:special-value}, together with \eqref{eq:theta-symmetry} and \eqref{eq:watson}, we note that
\begin{equation} \label{eq:watson-again}
\begin{aligned} 
&
\begin{aligned}
& \theta_3(2u_1 - u_2 + u_4,q)\theta_2(u_2+u_4,q) -
\theta_2(2u_1 - u_2 + u_4,q)\theta_3(u_2+u_4,q) \\ & \qquad =
\theta_2(u_1+u_4 + \textstyle \half ,q^{1/2})\theta_2(u_1-u_2+\textstyle\half,q^{1/2}) \\
& \qquad = -(\zeta_1^{1/2}\zeta_4^{1/2} - \zeta_1^{-1/2}\zeta_4^{-1/2})
(\zeta_1^{1/2}\zeta_2^{-1/2} - \zeta_1^{-1/2}\zeta_2^{1/2})
 q^{1/4} \gamma(q)^2,
\end{aligned} \\ 
&
\begin{aligned}
& \theta_3(2u_1-u_2+u_4,q)\theta_2(u_4-u_2,q) +
\theta_2(2u_1-u_2+u_4,q)\theta_3(u_4-u_2),q) \\
& \qquad = \theta_2(u_1-u_2+u_4,q^{1/2})\theta_2(u_1,q^{1/2})
\\ & \qquad
(\zeta_1^{1/2}\zeta_2^{-1/2}\zeta_4^{1/2}+\zeta_1^{-1/2}\zeta_2^{1/2}\zeta_4^{-1/2})
(\zeta_1^{1/2} + \zeta_1^{-1/2})  q^{1/4}\gamma(q)^2.
\end{aligned}
\end{aligned}
\end{equation} 

Let's work with respect to the following modified generators of the Floer cochain complexes (tensored with $\bC[[q]]$). We multiply the given generator of $\mathit{CF}^*(V_{1,\zeta_1},V_{2,\zeta_2})[[q]]$ by $1/\gamma$, and the generator of $\mathit{CF}^*(V_{1,\zeta_1},V_{4,\zeta_4})[[q]]$ with $\gamma$. So far, the outcome is that when written out in terms of the new generators, the formulae for \eqref{eq:123-product} and \eqref{eq:134-product} gets divided by $\gamma$ (hence become independent of $q$), whereas \eqref{eq:124-product} gets divided by $\gamma^2$. Finally, let's change the given generators of $\mathit{CF}^*(V_{2,\zeta_2},V_{4,\zeta_4})$ by applying the matrix
\begin{equation} \label{eq:basis-change}
\begin{aligned} &
\begin{pmatrix}
q^{-1/4}\theta_2(u_2+u_4,q) & q^{-1/4}\theta_2(u_4-u_2,q) \\
-\theta_3(u_2+u_4,q) & \theta_3(u_4-u_2,q)
\end{pmatrix} \\ & \qquad \qquad =
\begin{pmatrix}
\zeta_2^{1/2}\zeta_4^{1/2} + \zeta_2^{-1/2}\zeta_4^{-1/2} & \zeta_4^{1/2}\zeta_2^{-1/2} + \zeta_4^{-1/2}\zeta_2^{1/2} \\
-1 & 1 
\end{pmatrix} + O(q).
\end{aligned}
\end{equation}
Using \eqref{eq:watson-again}, one sees that after this further change of generators, \eqref{eq:134-product} turns into (the $q$-independent expression)
\begin{equation}
\begin{aligned}
& \big( -(\zeta_1^{1/2}\zeta_4^{1/2} - \zeta_1^{-1/2}\zeta_4^{-1/2})
(\zeta_1^{1/2}\zeta_2^{-1/2} - \zeta_1^{-1/2}\zeta_2^{1/2}), \\ & \qquad \qquad (\zeta_1^{1/2}\zeta_2^{-1/2}\zeta_4^{1/2}+\zeta_1^{-1/2}\zeta_2^{1/2}\zeta_4^{-1/2})
(\zeta_1^{1/2} + \zeta_1^{-1/2})
\big).
\end{aligned}
\end{equation}
There is a pitfall: the determinant of \eqref{eq:basis-change} is 
\begin{equation}
q^{-1/4} \theta_2(u_2,q^{1/2})\theta_2(u_4,q^{1/2}) = 
(\zeta_2^{1/2} + \zeta_2^{-1/2})(\zeta_4^{1/2} + \zeta_4^{-1/2}) \gamma(q)^2,
\end{equation}
hence the matrix is invertible only if $\zeta_2,\zeta_4 \neq -1$. However, one can replace $(u_2,u_4)$ in \eqref{eq:watson-again} by $(u_2+\frac{1}{3},u_4+\frac{1}{3})$ or $(u_2+\frac{2}{3},u_4+\frac{2}{3})$, and cover all cases in this way. To summarize, we have now shown that in suitable generators, the products \eqref{eq:123-product}--\eqref{eq:124-product} become constant in $q$. Because of associativity, it then follows automatically that \eqref{eq:234-product} is constant in $q$ as well (or one can check that by direct computation).

So far, we have been a bit vague about the formal context in which this computation takes place. What we want to do is to order the $12$ possible $V_{k,\zeta_k}$ nondecreasingly in $k$, and consider the associated directed $A_\infty$-category, which we denote by $\tilde\scrA_q$. Using the fact that the Floer cohomology between a given vanishing cycle with different choices of holonomy is zero, we can simplify the definition by setting 
\begin{equation}
\mathit{hom}_{\tilde{\scrA}_q}(V_{k,\zeta_k},V_{k,\zeta_k'}) = 0 \quad\text{if }\zeta_k \neq \zeta_k'. 
\end{equation}
Then, the argument above shows that $\tilde{\scrA}_q$ is trivial as an $A_\infty$-deformation. There is a natural action of $(\bZ/3)^2$ on this category, which tensors all objects with the restriction of line bundles on the torus having holonomies which are third roots of unity. After taking the semidirect product, we recover a category which is Morita equivalent to that associated to the cubic pencil on $\bC P^2$: more precisely,
\begin{equation} \label{eq:morita}
\tilde{\scrA}_q \rtimes (\bZ/3)^2 \iso \scrA_q \otimes \mathit{End}(\bC^3).
\end{equation}
This is the same trick mentioned in Remark \ref{th:z3-action} (we refer to \cite{seidel03b, seidel08b, sheridan11b}, where parallel arguments appear, for further explanation). The triviality of the deformation $\tilde{\scrA}_q$ implies its equivariant triviality (a general property of finite group actions in characteristic $0$) and, via \eqref{eq:morita}, the triviality of $\scrA_q$.

\subsection{The anticanonical pencil on the Hirzebruch surface $F_1$\label{subsec:f1}}
Our final example is the Lefschetz fibration obtained from the anticanonical Lefschetz pencil on $F_1$ ($\bC P^2$ blown up at one point). This can be obtained from its counterpart for $\bC P^2$ by filling in one of the punctures in the fibre $M$, which we choose as shown in Figure \ref{fig:12} ($\bar{M}$ remains the same, but $\delta M$ loses a point). To be precise, this ``filling in'' process changes the notion of exact Lagrangian submanifold in $M$, hence the vanishing cycles will move around slightly, but that will not affect our computation.

It will turn out that in this case, the deformation $\scrA_q$ is nontrivial already to first order in $q$. For that, consider the one-dimensional spaces
\begin{equation} \label{eq:one-dim}
\mathit{CF}^*(V_{k_1,t_1},V_{k_2,t_2})\;\; \text{with }
(k_1,k_2) \in \{(1,2), (1,3), (1,4), (2,3), (3,4)\}.
\end{equation}
The products
\begin{equation} \label{eq:36-products}
\begin{aligned}
& \mathit{CF}^*(V_{k_2,t_2},V_{k_3,t_3}) \otimes
\mathit{CF}^*(V_{k_1,t_1},V_{k_2,t_2}) \longrightarrow
\mathit{CF}^*(V_{k_1,t_1},V_{k_3,t_3}), \\[-.1em]
& \text{where all morphism spaces are as in \eqref{eq:one-dim}, and} 
\; t_3 \equiv t_1+t_2 \pm \textstyle\frac{1}{2} \; \mathrm{mod} \; 3,
\end{aligned}
\end{equation}
are all nonzero, and remain unchanged under deformation to first order in $q$. In contrast, the product
\begin{equation} \label{eq:nontrivial-product}
\mathit{CF}^*(V_{2,5/2},V_{3,3/2}) \otimes \mathit{CF}^*(V_{1,1/2},V_{2,5/2}) \longrightarrow
\mathit{CF}^*(V_{1,1/2},V_{3,3/2})
\end{equation}
deforms nontrivially to first order: it is a nonzero constant times $1-q + O(q^2)$ (the two relevant triangles are shaded in Figure \ref{fig:12}). 

Suppose that $\scrA_q$ was trivial to first order. This means that its first order part represents the trivial class in the Hochschild cohomology $\mathit{HH}^2(\scrA,\scrA)$. Hence, this part must be the boundary of a degree $1$ Hochschild cochain $\epsilon$. Concretely, $\epsilon$ consists of endomorphisms $\epsilon(V_{k_1,t_1}, V_{k_2,t_2})$ of each morphism space in $\scrA$. If one only considers only the spaces with \eqref{eq:one-dim}, $45$ coefficients of $\epsilon$ are involved. The fact that the products \eqref{eq:36-products} have trivial $q$ term imposes $36$ linear conditions. An explicit computation shows that $34$ of those conditions are linearly independent, leaving an $11$-dimensional space of choices for our coefficients. However, all those choices can be realized by boundaries of degree $0$ Hochschild cochains, which means $\epsilon(V_{k_1,t_1},V_{k_2,t_2}) = \delta(V_{k_1,t_1}) - \delta(V_{k_2,t_2})$. Hence, by adding a suitable boundary to $\epsilon$, we may assume that it vanishes on all the spaces \eqref{eq:one-dim}. But that contradicts the condition coming from \eqref{eq:nontrivial-product}, which is that
\begin{equation}
\epsilon(V_{1,1/2},V_{3,3/2}) - \epsilon(V_{1,1/2},V_{2,5/2}) - \epsilon(V_{2,5/2},V_{3,3/2}) \neq 0.
\end{equation}

\section{\label{sec:apply-theorem}Applications of Theorem \ref{th:main}}

\subsection{Monodromy considerations\label{subsec:monodromy}}
We will consider a class of examples that includes the Lefschetz pencils from Section \ref{sec:elementary}, but this time making use of Theorem \ref{th:main}. Start with a Lefschetz pencil on a del Pezzo surface $\roundbar{E}$. Blowing up its base locus yields a rational elliptic surface $\cornerbar{E}$ (topologically, $\bC P^2$ blown up at $9$ points). Let's write
\begin{equation} \label{eq:blowup-basis}
H_2(\cornerbar{E}) = \bZ L \oplus \bZ A_0 \oplus \cdots \oplus \bZ A_8,
\end{equation}
with the intersection form $\mathrm{diag}(1,-1,\dots,-1)$. One can choose the basis so that
\begin{align}
& [\bar{M}] = 3L - A_0 - \cdots - A_8, \\
& [\delta E|] = \begin{cases} A_0 + \cdots + A_{d-1} & \text{if $\roundbar{E}$ is $\bC P^2$ blown up at $9-d$ points,} \\
A_0 + \cdots + A_6 + (L - A_7 - A_8) & \text{if $\roundbar{E} \iso \bC P^1 \times \bC P^1$.}
\end{cases}
\end{align} 
Generally, we write $d$ for the number of connected components of $\delta E|$.

\begin{lemma} \label{th:monodromy}
There are diffeomorphisms of $\cornerbar{E}$, which preserve the deformation class of the symplectic form, and which realize the following automorphisms of homology: (i) any permutation of $\{A_0,\dots,A_8\}$; (ii) the reflection 
\begin{equation} \label{eq:reflection}
X \longmapsto X + (X \cdot S)S, \quad \text{where $S = L - A_6 - A_7 - A_8$.}
\end{equation}
\end{lemma}

\begin{proof}
For (i), one deforms the symplectic form so that $\cornerbar{E}$ is a blowup of $\bC P^2$ whose exceptional divisors have very small area. Then, a symplectic automorphism of $\bC P^2$ which  permutes the blowup points, and permutes local Darboux charts near those points, can be lifted to an automorphism of $\cornerbar{E}$. For (ii), one deforms the blowup points so that the last $3$ are collinear. Then, the blowup contains a $(-2)$-curve in class $S$, which one can contract to an ordinary double point singularity. The diffeomorphism we are looking for is the associated monodromy map (a Dehn twist).
\end{proof} 

\begin{lemma} \label{th:monodromy-applied}
Suppose that our Lefschetz fibration arises from a Lefschetz pencil on a del Pezzo surface which is not $\bC P^2$ blown up at $1$ or $2$ points. Then $\cornerbar{b} = 1$ is a solution of \eqref{eq:fundamental-equation}.
\end{lemma}

\begin{proof}
This will be a direct application of Lemma \ref{th:invariant-subspace}. The simplest case is where $\roundbar{E}$ is ${\bC P^2}$. Then, any diffeomorphism from Lemma \ref{th:monodromy}(i) preserves $[\delta E|]$; the subspace of classes that are invariant under such diffeomorphisms is clearly spanned by $[\delta E|]$ and $[\bar{M}]$.

Next, let $\roundbar{E}$ be $\bC P^2$ blown up at $9-d$ points, where $1 \leq d \leq 6$. Since $[\delta E|]$ should be preserved, when considering Lemma \ref{th:monodromy}(i), only permutations which preserve the subsets $\{A_0,\dots,A_{d-1}\}$ and $\{A_d,\dots,A_8\}$ can be used. With respect to such permutations, the invariant subspace of homology is three-dimensional: it is spanned by $[\delta E|]$, $[\bar{M}]$ and $A_d + \cdots + A_8$. However, since the class $S$ from \eqref{eq:reflection} satisfies $S \cdot [\delta E|] = 0$, the diffeomorphism from Lemma \ref{th:monodromy}(ii) also preserves $[\delta E|]$,  but not $A_d + \cdots + A_8$. Hence, after using that diffeomorphism as well, the subspace of invariant cohomology becomes two-dimensional.

For $\roundbar{E} \iso \bC P^1 \times \bC P^1$, one can use Lemma \ref{th:monodromy}(i) to narrow down the invariant part of homology to the subspace spanned by $[\delta E|]$, $[\bar{M}]$, and $A_7+A_8$. The last step is then as before.
\end{proof}

Theorem \ref{th:main} shows that in the situation of Lemma \ref{th:monodromy-applied}, the deformation $\scrA_q$ (with $\bar{b} = 1$) is trivial. This recovers the observations made in Sections \ref{subsec:cubic-pencil} and \ref{subsec:cp2-pencil}.

\subsection{Enumerative geometry}
We would like to complement the previous considerations by explicit computations of $z^{(1)}$, still assuming that $\cornerbar{b} = 1$. For that, we make use of known results on the genus zero enumerative geometry of the rational elliptic surface \cite{goettsche-pandharipande98, bryan-leung97}.

\begin{lemma}[\protect{Bryan-Leung \cite[Theorem 6.2]{bryan-leung97}}]
Take any class
\begin{equation} \label{eq:a-condition}
A \in H_2(\cornerbar{E}), \quad \bar{M} \cdot A = 1, \quad A \cdot A = 2k-1
\end{equation}
(note that if $\bar{M} \cdot A = 1$, then $A \cdot A$ is necessarily odd.) The count of rational curves in this class is $z_A = z_k \cdot A$, where the numbers $z_k$ are determined by the generating series
\begin{equation}
\sum_k z_k q^k = \prod_m \frac{1}{(1-q^m)^{12}} = \frac{q^{1/2}}{\Delta(q)^{1/2}}
= 1 + 12 q + O(q^2).
%
\end{equation}
\end{lemma}

It remains to organize the way in which we sum over all homology classes, and for this we follow \cite{hosono-saito-stienstra97}. All $A$ as in \eqref{eq:a-condition} can be written as
\begin{align} 
& A = A_0 + X - \textstyle{\half} (X \cdot X) [\bar{M}] + k [\bar{M}],  \label{eq:a-and-a0}
\text{ where $X$ satisfies}
\\ & \label{eq:lattice} \bar{M} \cdot X = 0, \quad A_0 \cdot X = 0.
\end{align}
Therefore,
\begin{equation} \label{eq:little-s}
\begin{aligned}
z^{(1)} = & \; \sum_{X,k} q^{\delta E| \cdot A_0 + \delta E| \cdot X - \frac{d}{2} X \cdot X + dk} z_k\, (A_0 + X)  \\  
& +\sum_{X,k} -{\textstyle\half}(X \cdot X) q^{\delta E| \cdot A_0 + \delta E| \cdot X - \frac{d}{2} X \cdot X + dk} z_k \, [\bar{M}]  \\
& + \sum_{X,k} q^{\delta E| \cdot A_0 + \delta E| \cdot X - \frac{d}{2} X \cdot X + dk} \, k z_k\, [\bar{M}] 
\\
= & \; q^{\delta E| \cdot A_0} \frac{q^{d/2}}{\Delta(q^d)^{1/2}}\sum_X  q^{\delta E| \cdot X - \frac{d}{2} X \cdot X} (A_0 + X) \\
& + q^{\delta E| \cdot A_0} \frac{q^{d/2}}{\Delta(q^d)^{1/2}} \sum_X -{\textstyle\half}(X \cdot X) q^{\delta E| \cdot X - \frac{d}{2} X \cdot X}\, [\bar{M}] \\
& + q^{\delta E| \cdot A_0} \frac{q}{d} \partial_q \Big( \frac{q^{d/2}}{\Delta(q^{d/2})} \Big)  \sum_X q^{\delta E| \cdot X - \frac{d}{2} X \cdot X}\, [\bar{M}].
\end{aligned}
\end{equation}

Suppose that we are in the situation of Lemma \ref{th:monodromy-applied}, so that \eqref{eq:fundamental-equation} holds. By taking the product of \eqref{eq:little-s} with $[\bar{M}]$, one sees that the function $\psi$ appearing in that equation satisfies
\begin{equation} \label{eq:determine-psi}
\frac{1}{\psi} = \frac{q}{d} \, q^{\delta E|\cdot A_0} \frac{q^{d/2}}{\Delta(q^{d/2})} \sum_X q^{\delta E| \cdot X - \frac{d}{2} X \cdot X}.
\end{equation}
The subspace of classes $X$ in \eqref{eq:lattice}, with the intersection form, is isomorphic to the (negative definite even) $E_8$ lattice, see e.g. \cite{donagi-grassi-witten96}. The sum in \eqref{eq:determine-psi} is a theta-function (usually an inhomogeneous one) for that lattice. Note also that $\psi$ determines $\eta$, by Lemma \ref{th:4d}.

\begin{example}
The evaluation of \eqref{eq:determine-psi} is particularly simple for $[\delta E|] = A_0$, which corresponds to the anticanonical pencil on the largest del Pezzo surface. One then has $\delta E| \cdot X = 0$ by \eqref{eq:lattice}, and therefore the sum in \eqref{eq:determine-psi} is the ordinary theta-function $\Theta_{E_8}(q)$ of the $E_8$ lattice. One concludes (compare \cite[Proposition 6.1]{hosono-saito-stienstra97}) that
\begin{equation}
\psi = \frac{q^{1/2} \Delta(q)^{1/2}}{\Theta_{E_8}(q)}.
\end{equation}
\end{example}

\begin{example} \label{th:cp2-psi-eta}
Take $\roundbar{E} = \bC P^2$. Then, the sum in \eqref{eq:determine-psi} is the inhomogeneous theta-function $\Theta_{E_8+s}(q^9)$, where $s$ is a ``shallow hole'' \cite[p.~121]{conway-sloane}. By \cite{hosono-saito-stienstra97,zagier98}, one has
\begin{equation}
\Theta_{E_8+s}(q) = 9 \, \frac{\Delta(q)^{1/2}}{q^{1/2}} \, \frac{q^{1/18}}{\Delta(q^{1/3})^{1/6}}.
\end{equation}
Applying this to \eqref{eq:determine-psi}, one gets
\begin{equation} \label{eq:psi-example}
\psi = 9 \frac{\Delta(q^9)^{1/2}}{q^{9/2}} \Theta_{E_8+s}(q^9)^{-1} = 
\frac{\Delta(q^3)^{1/6}}{q^{1/2}}.
\end{equation}
\end{example}

\subsection{The anticanonical pencil on $F_1$}
The Hirzebruch surface $\roundbar{E} = F_1$ is one of the examples excluded from Lemma \ref{th:monodromy-applied}. We already encountered it in Section \ref{subsec:f1}: the nontriviality result proved there, together with Theorem \ref{th:main}, implies that $\cornerbar{b} = 1$ can't be a solution of \eqref{eq:fundamental-equation}; in fact, there can't even be a solution of the form $\cornerbar{b} = 1 + O(q^2)$. To close the circle, let's re-derive that result by an explicit computation.

There are $8$ classes $X$ in \eqref{eq:lattice} for which the exponent in \eqref{eq:little-s} takes on its minimal value $-1$ (these are the classes $A_1,\dots,A_8$ of the connected components of $\delta E|$), and exactly one class where the exponent is zero, namely $A_0$. Hence,
\begin{equation} \label{eq:sq-quadratic}
z^{(1)} \equiv q^{-1}[\delta E|] + A_0 + O(q) \;\;
\text{mod} \; [\bar{M}].
\end{equation}
This shows that, as predicted by the general theory, $z^{(1)}$ does not lie in the subspace generated by $[\delta E|]$ and $[\bar{M}]$.

We know from Lemma \ref{th:solutions} that \eqref{eq:fundamental-solution} has a solution with nontrivial $\cornerbar{b}$. In view of \eqref{eq:sq-quadratic}, it is natural to try the ansatz
\begin{equation}
\cornerbar{b} = \beta(q)^{A_0}, \;\; \beta \in 1 + q\bC[[q]]. 
\end{equation}
By an explicit computation extending that in \eqref{eq:little-s}, one gets
\begin{equation}
\begin{aligned}
z^{(1)} & = 
\frac{(q^8 \beta(q) )^{\half}}{\Delta(q^8 \beta(q))^{\half}}
\sum_X (A_0 + X) \beta(q)^{- \half (X \cdot X) - 1} q^{\delta E| \cdot X - 4X \cdot X} \\
& = q^{-1} [\delta E|] + A_0 \beta(q)^{-1} +\\ 
& \qquad \qquad + q \beta(q) (-{\textstyle\frac{16}{3}} A_0 + {\textstyle\frac{5}{3}} [\delta E|]) + q^2 (\textstyle\frac{28}{3} A_0 + \textstyle\frac{7}{3} [\delta E|] )+ O(q^3) \;\; \mathrm{mod} \; [\bar{M}]. 
\end{aligned}
\end{equation}
Solving \eqref{eq:fundamental-equation} to the first few orders in $q$ yields
\begin{equation} \label{eq:beta-function}
\beta(q) = 1 + q - \textstyle\frac{8}{3} q^2 - q^3 + O(q^4).
\end{equation}

\begin{remark}
In principle, one should be able to check geometrically, using Figure \ref{fig:12}, that \eqref{eq:beta-function} leads to a deformation $\scrA_{q,\bar{b}}$ which is trivial to low orders. As in \eqref{eq:beta-factors}, this would mean counting holomorphic polygons that pass over the ``filled in'' puncture $m$ times with an additional factor $\beta(q)^m$. We won't try to do that, but one can at least note that \eqref{eq:nontrivial-product} would become $\beta(q) - q + O(q^2) = 1 + O(q^2)$. Hence, our original argument for non-triviality of the deformation $\scrA_q$ breaks down if one turns on the appropriate bulk term (as it must).
\end{remark}

\section{The elliptic pencil on $\bC P^2$, revisited\label{sec:example-revisited}}

\subsection{The differential equation\label{subsec:example-connection}}
We continue our discussion of Example \ref{th:cp2-psi-eta}, with the aim of testing the ideas from Section \ref{subsec:conjecture}. As before, we set $\cornerbar{b} = 1$ throughout, and will omit all bulk terms from the notation. 

Fix $0 \leq i < j \leq 8$. Lemma \ref{th:monodromy} says that we can find a diffeomorphism of $\cornerbar{E}$ which preserves the deformation class of the symplectic form, reverses $S_{ij} = A_i - A_j$, and leaves its orthogonal complement (under the intersection form) unchanged. Therefore, we necessarily have
\begin{equation}
[\bar{M}] \ast S_{ij} = \lambda S_{ij}
\end{equation}
for some $\lambda \in \bC((q))$. By the same kind of argument, $\lambda$ is the same for all $(i,j)$. Explicitly,
\begin{equation} \label{eq:lambda-function}
\lambda = {\textstyle -\half} ([\bar{M}] \ast S_{ij}) \cdot S_{ij} = {\textstyle -\half} \sum_{\bar{M} \cdot A = 1} q^{\delta E| \cdot A} z_k (A \cdot S_{ij})^2,
\end{equation}
where $k = \half(A \cdot A + 1)$, as in \eqref{eq:a-and-a0}. 
For classes $A$ with $\bar{M} \cdot A = 1$, 
\begin{equation}
\sum_{i<j} (A \cdot S_{ij})^2 = -9 A \cdot A + 2 \delta E| \cdot A + 1 = -18k + 2\, \delta E| \cdot A + 10.
\end{equation}
Hence, by summing up \eqref{eq:lambda-function} over all $i<j$, one gets
\begin{equation}
\lambda = {\textstyle \frac14} \sum_A q^{\delta E| \cdot A} z_k k  -
{\textstyle \frac{1}{36}} \sum_A q^{\delta E| \cdot A} z_k (\delta E| \cdot A) -
{\textstyle \frac{5}{36}} \sum_A q^{\delta E| \cdot A} z_k.
\end{equation}
All three summands can be identified as before, which yields
\begin{equation}
\begin{aligned}
\lambda & = {\textstyle \frac{1}{36}} \partial_q \Big( \frac{q^{9/2}}{\Delta(q^9)^{1/2}} \Big) \Theta_{E_8+s}(q^9) - \frac{\eta}{4\psi} - \frac{1}{q\psi} \\
 & = {\textstyle \frac{1}{4}} \partial_q \Big( \frac{q^{9/2}}{\Delta(q^9)^{1/2}} \Big) \frac{\Delta(q^9)^{1/2}}{q^{9/2}} \frac{q^{1/2}}{\Delta(q^3)^{1/6}} - {\textstyle\frac{1}{4}} \partial_q \Big( \frac{q^{1/2}}{\Delta(q^3)^{1/6}}\Big) - \frac{1}{q^{1/2} \Delta(q^3)^{1/6}}.
\end{aligned}
\end{equation}

We now apply an argument of the same kind as Lemma \ref{th:eigenvalues}, but in reverse direction, using our knowledge of the eigenvalue $\lambda$ to determine the enumerative invariant $z_q^{(2)}$.
Namely, because of the associativity of the quantum product, the following two expressions must be equal:
\begin{align}
& [\bar{M}] \ast ([\bar{M}] \ast S_{ij}) = \lambda^2 S_{ij},  \\
& ([\bar{M}] \ast [\bar{M}]) \ast S_{ij} = \Big(\frac{1}{q\psi} [\delta E|] + \frac{\eta}{\psi} [\bar{M}] + 4z_q^{(2)}\Big) \ast S_{ij} = \Big( \partial_q \Big(\frac{\lambda}{\psi}\Big)+ 4 z_q^{(2)} \Big) S_{ij} \label{eq:ss}
\end{align}
In \eqref{eq:ss}, we have used the divisor axiom as well as Lemma \ref{th:4d}. This leads to the relation
\begin{equation} \label{eq:point-recursion-relation}
z_q^{(2)} = {\textstyle \frac{1}{4}} \big( \lambda^2 - \partial_q (\lambda/\psi) \big).
\end{equation}

At this point, we have determined $\psi$, $\eta$ and $z^{(2)}_q$, hence know the coefficients in the matrix $\Gamma$ from \eqref{eq:fundamental-solution}. One could now try to solve for $\Theta$ in closed form. We settle for a more modest goal, namely the first few Taylor coefficients, which can be computed by solving \eqref{eq:fundamental-solution} order by order in $q$:
\begin{equation} \label{eq:compute-fundamental-solution}
\begin{aligned}
&
\Theta_{11} =
1 + 6 q^3 + 6 q^9 + 6 q^{12} + 12 q^{21} + 6 q^{27} + O(q^{32}), \\
&
\Theta_{21} = 
-18 q^2 - 72 q^5 - 306 q^8 - 1008 q^{11} - 2934 q^{14} - 7704 q^{17} - 19134 q^{20} \\ 
& \qquad \qquad - 44496 q^{23} - 99270 q^{26} - 212256 q^{29} + O(q^{32}), \\
&
\Theta_{12} =
-q - q^4 - 2 q^7 - 2 q^{13} - q^{16} - 2 q^{19} - q^{25} - 2 q^{28} - 2 q^{31} + O(q^{32}), \\
&
\Theta_{22} =
1 + 8 q^3 + 44 q^6 + 152 q^9 + 487 q^{12} + 1352 q^{15} + 3518 q^{18} + 8480 q^{21} \\
& \qquad \qquad + 19503 q^{24} + 42768 q^{27} + 90530 q^{30} + O(q^{32}).
\end{aligned}
\end{equation}

\begin{remark} \label{th:oberdieck}
With the help of the {\em Online Encyclopedia of Integer Sequences}, the formulae above suggest the following candidates for closed form solutions: 
\begin{equation}
\begin{aligned}
& \Theta_{11} = \Theta_{\mathit{hex}}(q^3), \\ 
& \Theta_{12} = -{\textstyle\frac{1}{3}} \Theta_{\mathit{hex}+d}(q^3) = -\frac{\Delta(q^9)^{1/8}}{3\Delta(q^3)^{1/24}}.
\end{aligned}
\end{equation}
where $\Theta_{\mathit{hex}}$ is the theta-function of the hexagonal plane lattice, and $\Theta_{\mathit{hex}+d}$ its inhomogeneous counterpart with respect to a ``deep hole'' \cite[p.~111]{conway-sloane}.
\end{remark}

\subsection{The mirror map\label{subsec:mirror-map}}
As mentioned in Section \ref{subsec:cp2-pencil}, the mirror of the cubic pencil on $\bC P^2$ is the extremal rational elliptic surface $X = X_{9111}$. In fact, for a suitable choice of coordinates on the base, its elliptic fibration
\begin{equation}
p: X \longrightarrow \bC P^1 = \bC \cup \{\infty\}
\end{equation}
is a fibrewise compactification of \eqref{eq:toric-mirror-cp2}. This description makes it easy to determine the $j$-invariant of the fibre $X_z = p^{-1}(z)$, which is
\begin{equation}
j(z) = z^3\frac{(z^3-24)^3}{z^3 - 27}.
\end{equation}
Following the prediction in Remark \ref{th:mirror-fibre}, let's substitute
\begin{equation} \label{eq:my-mirror-map}
\begin{aligned}
z & = -\frac{\Theta_{11}}{\Theta_{12}} = q^{-1} + 5 q^2 - 7 q^5 + 3 q^8 + 15 q^{11} - 32 q^{14} + 9 q^{17} \\ & \qquad\qquad\qquad\qquad\qquad   + 58 q^{20} - 96 q^{23} + 22 q^{26} + 149 q^{29} + O(q^{30}),
\end{aligned}
\end{equation}
where $\Theta_{11}$, $\Theta_{12}$ are the fundamental solutions from \eqref{eq:compute-fundamental-solution}. Then, to the available precision,
\begin{equation} \label{eq:j9}
j(z) = q^{-9} + 744 + 196884 q^9 + 21493760 q^{18} + O(q^{23})
\end{equation}
agrees with the classical $j$-function (with parameter $q^9$). This is precisely what one expects based on Conjecture \ref{th:get-fibre}: in informal language as in \eqref{eq:log-area}, the fibre $\bar{M}$ is a two-torus with ``area'' $-\log(q^9)$, because there are $9$ punctures; the mirror of which is precisely the elliptic curve with $j$-invariant \eqref{eq:j9}.

\begin{remark}
To relate this to previous computations in the literature, consider the family of cubic curves in $\bC P^2$ given by
\begin{equation} \label{eq:hesse}
x_0^3 + x_1^3 + x_2^2 - \tilde{z}^{-1/3} x_0x_1x_2 = 0,
\end{equation}
for which the mirror map (see e.g.\ \cite{roan95}) is
\begin{equation} \label{eq:classical-mirror-map}
\tilde{z} = \tilde{q} - 15 \tilde{q}^2 + 171 \tilde{q}^3 - 1679 \tilde{q}^4 + 15054 \tilde{q}^5 - 126981 \tilde{q}^6 + \cdots 
\end{equation}
The curves \eqref{eq:hesse} are $\bZ/3$-covers of the (compactified) fibres in \eqref{eq:toric-mirror-cp2}, where $\tilde{z} = z^{-3}$. Correspondingly, their mirror is $\bZ/3$-quotient of the cubic pencil. Indeed, if one also sets $\tilde{q} = q^{1/3}$, \eqref{eq:my-mirror-map} turns into \eqref{eq:classical-mirror-map}.
\end{remark}

\appendix
\section{Some Maurer-Cartan theory\label{sec:algebra}}

\appsubsection{One-parameter formal deformations}
We recall some elements of abstract Maurer-Cartan theory (the original reference is \cite{goldman-millson88}; an exposition close to our purpose is \cite{lunts10}). Let $\frakg$ be a dg Lie algebra over $\bC$. One considers solutions $\alpha_q \in q \frakg^1[[q]]$ (formal power series with coefficients in $\frakg^1$ and vanishing constant term) of the Maurer-Cartan equation
\begin{equation} \label{eq:maurer-cartan}
d\alpha_q + \textstyle\half [\alpha_q,\alpha_q] = 0.
\end{equation}
Take the group associated to the (pronilpotent) Lie algebra $q \frakg^0[[q]]$. We denote its elements by $\exp(\gamma_q)$, for $\gamma_q \in q \frakg^0[[q]]$. This group acts on the set of Maurer-Cartan elements by gauge transformations:
\begin{equation} \label{eq:gauge}
\Gamma_{\exp(\gamma_q)}(\alpha_q) = \alpha_q + ([\gamma_q,\alpha_q] - d\gamma_q) + \text{(terms of order $\geq 2$ in $\gamma$).}
\end{equation}
We say that $\alpha_q$ is trivial if it is gauge equivalent to zero. If we write $\alpha_q = q\alpha_1 + q^2\alpha_2 + \cdots$, then it follows from \eqref{eq:maurer-cartan} that the first order piece satisfies $d\alpha_1 = 0$, and from \eqref{eq:gauge} that the class
\begin{equation}
[\alpha_1] \in H^1(\frakg)
\end{equation}
is gauge invariant. There are standard rigidity and unobstructedness results:

\begin{lemma}
Suppose that $H^1(\frakg) = 0$. Then any Maurer-Cartan element is trivial.
\end{lemma}

\begin{lemma}
Suppose that $H^2(\frakg) = 0$. Then, any class in $H^1(\frakg)$ can be realized as the first order piece of a Maurer-Cartan element.
\end{lemma}

Any Maurer-Cartan element defines a deformed differential on $\frakg[[q]]$, namely
\begin{equation} \label{eq:deformed-differential}
d_{\alpha_q} = d + [\alpha_q,\cdot].
\end{equation}
There is a distinguished cohomology class, the Kaledin class
\begin{equation} \label{eq:kaledin-class}
\kappa(\alpha_q) = [\partial_q \alpha_q] \in H^1(\frakg[[q]],d_{\alpha_q}).
\end{equation}
The adjoint action of $\exp(\gamma_q)$ relates $d_{\alpha_q}$ and $d_{\Gamma_{\exp(\gamma_q)}(\alpha_q)}$, hence induces isomorphisms
\begin{equation} \label{eq:conjugation-isomorphisms}
H^*(\frakg[[q]],d_{\alpha_q}) \iso H^*(\frakg[[q]], d_{\Gamma_{\exp(\gamma_q)}(\alpha_q)}).
\end{equation}
The Kaledin class is gauge invariant, which means that it is preserved under \eqref{eq:conjugation-isomorphisms} \cite[Proposition 7.3(a)]{lunts10}. In that sense, it is an invariant of the gauge equivalence class of $\alpha_q$.

\begin{lemma}[\protect{\cite[Proposition 7.3(b)]{lunts10}}] \label{th:general-kaledin}
A Maurer-Cartan element is trivial if and only if its Kaledin class vanishes.
\end{lemma}
%
%

\begin{application}
Given an $A_\infty$-algebra (or curved $A_\infty$-algebra) $\scrA$, one can consider the shifted Hochschild cochain space
\begin{equation}
\frakg = \mathit{CC}^*(\scrA,\scrA)[1] = \prod_{d \geq 0} \mathit{Hom}^{*+1-d}(\scrA^{\otimes d},\scrA),
\end{equation}
which is a dg Lie algebra with the Gerstenhaber bracket, and with the differential determined by $\mu^*_{\scrA}$. The corresponding deformation problem is that of classifying formal $A_\infty$-deformations $\scrA_q$ (which, again, may have a curvature term). The complex $(\mathit{CC}^*(\scrA,\scrA)[[q]],d_{\alpha_q})$ is the Hochschild complex of the deformed structure $\scrA_q$, and the element \eqref{eq:kaledin-class} is \eqref{eq:dq-class}. There is also a straightforward generalization to $A_\infty$-categories. Hence, Lemma \ref{th:deformation-is-trivial} is an instance of Lemma \ref{th:general-kaledin}.
\end{application}

\appsubsection{The filtered theory}
One variation on this story is to consider a dg Lie algebra $\frakg$ which is pronilpotent, which means that it comes with a complete decreasing filtration
\begin{equation} \label{eq:pronilpotent}
\begin{aligned}
& \frakg = F^{(1)} \frakg \supset F^{(2)} \frakg \supset \cdots, \\
& d(F^{(i)} \frakg) \subset F^{(i)} \frakg, \\
& [F^{(i)} \frakg, F^{(j)} \frakg] \subset F^{(i+j)} \frakg.
\end{aligned}
\end{equation}
One can then consider solutions $\alpha$ of \eqref{eq:maurer-cartan} in $\frakg^1$ itself, and let the group associated to $\frakg^0$ act on them. Essentially, the filtration takes over the role (ensuring convergence) previously played by the formal parameter $q$. The leading part of a deformation is now the image $\alpha^{(1)}$ of $\alpha$ in $\frakg/F^{(2)}\frakg$, which gives rise to a gauge invariant class
\begin{equation} \label{eq:f2-quotient}
[\alpha^{(1)}] \in H^1(\frakg/F^{(2)}\frakg).
\end{equation}
However, where before, the same cohomology groups were responsible for deformation theory at any order in $q$, we now encounter all the different groups $H^*(F^{(i)}\frakg/F^{(i+1)}\frakg)$. The basic rigidity and unobstructedness results are:

\begin{lemma} \label{th:linear-def-2}
Suppose that
\begin{equation} \label{eq:degree-restriction-2}
H^1(F^{(i)} \frakg/ F^{(i+1)} \frakg) = 0 \quad \text{for all $i \geq 2$.}
\end{equation}
Then, two Maurer-Cartan elements in $\frakg$ are equivalent if and only if they yield the same class \eqref{eq:f2-quotient}.
\end{lemma}

\begin{lemma} \label{th:linear-def-1}
Suppose that
\begin{equation} \label{eq:degree-restriction-1}
H^2(F^{(i)} \frakg/ F^{(i+1)} \frakg) = 0 \quad \text{for all $i \geq 2$.}
\end{equation}
Then, any class in $H^1(\frakg/F^{(2)}\frakg)$ can be realized as the leading term of a Maurer-Cartan element in $\frakg$.
\end{lemma}

In the present setup, one can't generally expect to have an analogue of the Kaledin class. However, there is a (fairly obvious) subclass where such an analogue can be constructed. Namely, suppose that our filtration has a splitting compatible with all the given structure, meaning that
\begin{equation} \label{eq:bigraded}
\begin{aligned}
& \frakg = \prod_{i \geq 1} \frakg^{(i)}, \\
& d(\frakg^{(i)}) \subset \frakg^{(i)}, \\
& [\frakg^{(i)},\frakg^{(j)}] \subset \frakg^{(i+j)}.
\end{aligned}
\end{equation}
The filtration is $F^{(i)} = \prod_{j \geq i} \frakg^{(j)}$, and we still require that \eqref{eq:pronilpotent} holds (or equivalently, that differential and bracket are continuous with respect to the inverse limit topology). In that case, expanding $\alpha = \alpha^{(1)} + \alpha^{(2)} + \cdots$, one can define
\begin{equation} \label{eq:new-kaledin}
\kappa(\alpha) = [\alpha^{(1)} + 2\alpha^{(2)} + \cdots] \in H^*(\frakg,d_\alpha).
\end{equation}
For this class, an analogue of Lemma \ref{th:general-kaledin} holds. In fact, one can artificially introduce a formal parameter (which we denote by $h$, to distinguish it from similar developments later on) and write $\alpha_h = h\alpha^{(1)} + h^2\alpha^{(2)} + \cdots$. This reduces the situation to the previous one, with the added constraint that we have to preserve the additional grading on $\frakg[[h]]$ obtained by starting with the decomposition \eqref{eq:bigraded} and giving $h$ degree $-1$.

\begin{application}
Let $\scrA$ be a graded associative algebra. Take $\frakg \subset \mathit{CC}^*(\scrA,\scrA)[1]$ to be the subspace of Hochschild cochains which decrease degrees (with respect to the grading of $\scrA$). We equip this with the Gerstenhaber bracket, and with the differential obtained from the algebra structure of $\scrA$. This is a pronilpotent dg Lie algebra, which classifies $A_\infty$-extensions of the algebra structure (by $\mu^d$ terms, $d \geq 3$). Note that this is actually an instance of \eqref{eq:bigraded}, where $\frakg^{(i)}$ consists of Hochschild cochains of degree $-i$ (again, with respect to the grading of $\scrA$).

Consider the Euler derivation of $\scrA$, which multiplies each element by its degree. This is an element $\epsilon \in \mathit{CC}^1(\scrA,\scrA)$, and with respect to the twisted differential $d_\alpha$ associated to an $A_\infty$-extension of the algebra structure, it satisfies
\begin{equation}
d_{\alpha}\epsilon = \alpha^{(1)} + 2\alpha^{(2)} + \cdots,
\end{equation}
which is \eqref{eq:new-kaledin}. Since $\epsilon$ itself does not lie in $\frakg$, this does not show that the Kaledin class always vanishes; instead, we see that the Kaledin class vanishes if and only if there is $d_{\alpha}$-cocycle in the space $\epsilon + \frakg$. This explains how the general theory produces the formality criterion from \cite[Remark 7.6]{seidel-solomon10}.
\end{application}

\begin{application} \label{th:divisor-cl}
We consider the classification theory of noncommutative divisors with fixed $\scrA$ and $\scrP$ (see Section \ref{subsec:nc-divisors}). Take $\scrB$ as in \eqref{eq:b-from-a}, and equip it with a weight grading as before. Define $\frakg \subset \mathit{CC}^*(\scrB,\scrB)[1]$ to be the subspace of Hochschild cochains which strictly increase weight. The filtration by weight makes $\frakg$ into a pronilpotent dg Lie algebra, and in fact into an instance of \eqref{eq:bigraded}. The cohomology of the pieces $\frakg^{(i)}$ is determined by the long exact sequence \cite[Section 2c]{seidel14b}
\begin{equation} \label{eq:divisor-les}
\begin{aligned}
& \cdots \rightarrow H^{j-2i}(\mathit{hom}_{(\scrA,\scrA)}(\scrP^{\otimes_{\scrA} i},\scrA)) \\
& \qquad \qquad \qquad \qquad\longrightarrow H^j(\frakg^{(i)}) \longrightarrow H^{j-2i+1}(\mathit{hom}_{(\scrA,\scrA)}(\scrP^{\otimes_{\scrA} i},\scrA)) \rightarrow \cdots 
\end{aligned}
\end{equation}
Suppose that \eqref{eq:no-negative} holds. Then $\frakg$ satisfies \eqref{eq:degree-restriction-1} as well as \eqref{eq:degree-restriction-2}. Moreover, from \eqref{eq:divisor-les} for $i = 1$ we then get
\begin{equation} \label{eq:i-1-map}
0 \rightarrow H^1(\frakg^{(1)}) \longrightarrow H^0(\hom_{(\scrA,\scrA)}(\scrP,\scrA)) \longrightarrow
H^0(\hom_{(\scrA,\scrA)}(\scrP,\scrA)) \rightarrow \cdots
\end{equation}
Applying Lemmas \ref{th:linear-def-2} and \ref{th:linear-def-1}, one sees that noncommutative divisors are classified (up to isomorphism) by a certain subspace of $H^0(\hom_{(\scrA,\scrA)}(\scrP,\scrA))$, which is the kernel of the map in \eqref{eq:i-1-map}. In particular, the uniqueness part yields Lemma \ref{th:classification-theory-of-divisors}.
\end{application}

\appsubsection{A combined setup}
Take a pronilpotent Lie algebra $\frakg$ as in \eqref{eq:pronilpotent}. However, this time we want to consider solutions $\alpha_q$ of \eqref{eq:maurer-cartan} in $\frakg^1[[q]]$, modulo the action of exponentiated elements of $\frakg^0[[q]]$. The leading order term (with respect to the filtration) of such a Maurer-Cartan element gives rise to a class
\begin{equation} \label{eq:q-leading}
[\alpha_q^{(1)}] \in H^1(\frakg/F^{(2)}\frakg)[[q]].
\end{equation}
The analogues of Lemmas \ref{th:linear-def-2} and \ref{th:linear-def-1} are:

\begin{lemma} \label{th:q-def-2}
Suppose that \eqref{eq:degree-restriction-2} holds. Then, two Maurer-Cartan elements in $\frakg[[q]]$ are equivalent if and only if they yield the same class \eqref{eq:q-leading}.
\end{lemma}

\begin{lemma} \label{th:q-filtered}
Suppose that \eqref{eq:degree-restriction-1} holds. Then, any class in $H^1(\frakg/F^{(2)}\frakg)[[q]]$ can be realized by a Maurer-Cartan element in $\frakg[[q]]$.
\end{lemma}

In this framework, one can meaningfully restrict attention to certain kinds of subalgebras of $\bC[[q]]$. For $\frakg$ as in \eqref{eq:bigraded}, suppose that a $\bC$-linear subspace $V \subset \bC[[q]]$ is given. Defining its symmetric products as in \eqref{eq:v-subalgebra}, take the dg Lie algebra
\begin{equation}
\frakg[V] = \prod_i \frakg^{(i)} \otimes S^i[V] \subset \frakg[[g]].
\end{equation}
One can then consider solutions of the Maurer-Cartan equation in $\frakg[V]^1$, and there is a natural action of exponentiated elements of $\frakg[V]^0$. The primary invariant of a Maurer-Cartan element is the class
\begin{equation} \label{eq:v-leading}
[\alpha_q^{(1)}] \in H^1(\frakg^{(1)}) \otimes V = H^1(\frakg/F^{(2)}\frakg) \otimes V.
\end{equation}

\begin{lemma}
Suppose that \eqref{eq:degree-restriction-2} holds. Then, two Maurer-Cartan elements in $\frakg[V]$ are equivalent if and only if they yield the same class \eqref{eq:v-leading}.
\end{lemma}

\begin{lemma} \label{th:existence-in-v}
Suppose that \eqref{eq:degree-restriction-1} holds. Then, any class in $H^1(\frakg/F^{(2)}\frakg) \otimes V$ can be realized by a Maurer-Cartan element in $\frakg[V]$.
\end{lemma}

\begin{application}
Take $\frakg$ as in Application \ref{th:divisor-cl}. Then, $\frakg[[q]]$ describes the theory of deformations of noncommutative divisors. If \eqref{eq:no-negative} holds, one can apply Lemma \ref{th:existence-in-v}. In view of \eqref{eq:i-1-map}, this implies Lemma \ref{th:v-restriction}.
\end{application}

\appsubsection{Fields of definition for rigid objects}
When looking at Conjecture \ref{th:3} and Remark \ref{th:split-generation} in combination, the following question comes up. Consider the Fukaya category of a closed Calabi-Yau manifold, defined over $\overline{\bC((q))}$. Suppose that we have a split-generating full subcategory, which is defined over some sub-field of $\overline{\bC((q))}$. What are the implications of that fact for objects of the Fukaya category (not belonging to that subcategory)? Clearly, there is very little one can say about a general object, which might be a Lagrangian submanifold with any $\overline{\bC((q))}$-local coefficient system. Instead, we should restrict to objects that are rigid (have no nontrivial degree $1$ endomorphisms). This type of question is a classical one in algebraic geometry, and we can use the arguments developed there to obtain a satisfactory answer.

The general situation is this: consider algebraically closed fields $\bK \subset \tilde{\bK}$. Let $\scrB$ be an $A_\infty$-category over $\bK$. We consider the associated triangulated enlargement, the category $\Tw(\scrB)$ of twisted complexes, and its (chain level) idempotent completion $\Tw^{\pi}(\scrB)$ (this is $\Pi \Tw(\scrB)$ in the notation of \cite[Section 4c]{seidel04}; it is quasi-equivalent to the category of perfect modules over $\scrB$). Let's extend coefficients to $\tilde{\bK}$, forming $\tilde{\scrB} = \scrB \otimes_{\bK} \tilde{\bK}$. There is a commutative diagram
\begin{equation} \label{eq:k-k-embedding}
\xymatrix{
\Tw(\scrB) \otimes_{\bK} \tilde{\bK} \ar[d] \ar[rr] && \Tw(\tilde{\scrB}) \ar[d] \\
\Tw^{\pi}(\scrB) \otimes_{\bK} \tilde{\bK} \ar[rr] && \Tw^{\pi}(\tilde{\scrB}),
}
\end{equation}
where the horizontal arrows are full and faithful (in fact, inclusions of full $A_\infty$-subcategories), and the vertical arrows are cohomologically full and faithful.

\begin{lemma} \label{th:artin}
Suppose that $\scrB$ is proper. Let $\tilde{P}$ be an object of $\Tw^{\pi}(\tilde{\scrB})$ which is rigid, meaning that 
\begin{equation}
H^1(\hom_{\Tw^\pi(\tilde{\scrB})}(\tilde{P},\tilde{P})) = 0. 
\end{equation}
Then $\tilde{P}$ is ``defined over $\bK$ up to quasi-isomorphism'', which means that there is some object $P$ of $\Tw^{\pi}(\scrB)$ whose image under the horizontal map in \eqref{eq:k-k-embedding} is quasi-isomorphic to $\tilde{P}$.
\end{lemma}

\begin{proof}
Without loss of generality, we may assume that $\scrB$ is strictly proper (has chain level morphism spaces which are of finite total dimension). To simplify the notation, we also find it convenient to assume that $\scrB$ itself already admits shifts (translations). 

Our object $\tilde{P}$ is a direct summand (homotopy retract) of a twisted complex built from objects $X_1,\dots,X_m$ of the original $A_\infty$-category $\tilde{\scrB}$ (in that order). We want to consider ``all possible direct summands of twisted complexes built from the same pieces''. To implement this concretely, take the additive enlargement $\Sigma\scrB$ from \cite[Section 3c]{seidel04}, which allows finite direct sums. In it consider the object $C = \bigoplus_i X_i$. To make it into a twisted complex, we should equip it with a differential
\begin{equation}
d \in \mathit{hom}_{\Sigma\scrB}^1(C,C)
\end{equation}
which satisfies two conditions. One is strict upper triangularity:
\begin{equation} \label{eq:upper-triangular}
d = (d_{ji}), \quad d_{ji} \in \mathit{hom}_{\scrB}^1(X_i,X_j), 
\quad d_{ji} = 0 \text{ if } j \geq i.
\end{equation}
The second is the generalized Maurer-Cartan equation
\begin{equation} \label{eq:generalized-mc}
\mu^1_{\Sigma\scrB}(d) + \mu^2_{\Sigma\scrB}(d,d) + \cdots = 0.
\end{equation}
Writing out \eqref{eq:generalized-mc} shows that it is a system of polynomial equations (of degree $\leq m-1$) for the coefficients of $d$. Next, we need an endomorphism of the twisted complex $(C,d)$, which is idempotent up to chain homotopy. The necessary data are
\begin{align}
& p \in \mathit{hom}_{\Sigma\scrB}^0(C,C), \\
& h \in \mathit{hom}_{\Sigma\scrB}^{-1}(C,C).
\end{align}
The required conditions 
\begin{align}
& \mu^1_{\Tw(\scrB)}(p) = 0, \\
& \mu^1_{\Tw(\scrB)}(h) = p - \mu^2_{\Tw(\scrB)}(p,p),
\end{align}
can be written out as
\begin{align} \label{eq:pi-cocycle}
& \mu^1_{\Sigma\scrB}(p) + \mu^2_{\Sigma\scrB}(d,p) + \mu^2_{\Sigma\scrB}(p,d) + \cdots = 0, \\
& \label{eq:h-homotopy}
\begin{aligned}
& \mu^1_{\Sigma\scrB}(h) + \mu^2_{\Sigma\scrB}(d,h) + \mu^2_{\Sigma\scrB}(h,d) + \cdots \\
& \qquad = p - \mu^2_{\Sigma\scrB}(p,p) - \mu^3_{\Sigma\scrB}(d,p,p) - \cdots
\end{aligned}
\end{align}
The conditions \eqref{eq:generalized-mc}, \eqref{eq:pi-cocycle}, \eqref{eq:h-homotopy} are a finite set  polynomial equations for the coefficients of $d$, $p$ and $h$ (of degrees $\leq m+1$). They define an affine scheme over $\bK$, which we denote by $\scrU$. To every $\bK$-point of $\scrU$, one can associate an object of $\Tw^{\pi}(\scrB)$. We denote all such objects indiscriminately by $P$. $\scrU$ carries a bounded chain complex of (trivial) vector bundles, whose fibre is $\mathit{hom}_{\Sigma\scrB}(C,C)$, and whose differential is given by $\mu^1_{\Tw(\scrB)}$ at any point. The corresponding cohomology sheaves carry two idempotent endomorphisms, given by left and right composition with $p$ at any point. Hence, we get a coherent sheaf $\scrE$ on $\scrU$, whose fibre at any $\bK$-point is 
\begin{equation} \label{eq:h1-sheaf}
[p] H^1(\mathit{hom}_{\Tw(\scrB)}(C,C)) [p] \iso H^1(\mathit{hom}_{\Tw^\pi(\scrB)}(P,P)).
\end{equation}
Let $\scrS$ be the subscheme which is the support of this sheaf. Denote by $\scrU_l$ the irreducible components of $\scrU$, and by $\scrS_l \subset \scrU_l$ the piece of $\scrS$ which lies in $\scrU_l$.

Now consider the same construction over $\tilde{\bK}$. Because $\bK$ is already algebraically closed, $\tilde{\scrU}_l = \scrU_l \times_{\bK} \tilde{\bK}$ are the irreducible components of $\tilde{\scrU} = \scrU \times_{\bK} \tilde{\bK}$ \cite[Corollaire 4.4.5]{ega4-2}. Similarly, the support of the correspondingly defined sheaf $\tilde{\scrE}$ is $\tilde{\scrS} = \scrS \times_{\bK} \tilde{\bK}$. Each $\tilde{\bK}$-point of $\tilde{\scrU}$ now gives an object of $\Tw^{\pi}(\tilde{\scrB})$. Moreover, because of rigidity, the quasi-isomorphism type of this object is locally constant on $\tilde{\scrU} \setminus \tilde{\scrS}$, hence constant within each $\tilde{\scrU}_l \setminus \tilde{\scrS}_l$. 

Our original $\tilde{P}$ gives a point in $\tilde{\scrU}_l$, for some $l$, which does not lie in $\tilde{\scrS}_l$. Again using the fact that $\bK$ is algebraically closed, it follows that there is a point in $\scrU_l$ which does not lie in $\scrS_l$ (by an application of \cite[Corollaire 10.4.8]{ega4-3}). That point defines the required $P$.
\end{proof}

\begin{remark}
This argument uses a crude approximation of moduli stacks of objects \cite{toen-vaquie07}. The much more complicated question of fields of definition of $A_\infty$-categories has been studied in \cite{toen08}, but the results there require smoothness assumptions.
\end{remark}


\end{document}